%% file: AL-Semilinear.tex
\documentclass[11pt]{article}

\usepackage[a4paper,left=3cm,right=3cm,top=2cm,bottom=4cm,bindingoffset=5mm]{geometry}
\usepackage[utf8]{inputenc} 	% erm\"oglich die direkte Eingabe der Umlaute
\usepackage[T1]{fontenc} 	    % das Trennen der Umlaute
\usepackage{verbatim}   		% for using comment-environment

\usepackage{xspace}   %dynamisches Leerzeichen

\usepackage{amsmath}
\usepackage{amsthm}
\usepackage{amsfonts}
\usepackage{mathrsfs}
\usepackage{amssymb}

\usepackage[usenames,dvipsnames]{xcolor}
\usepackage{pgfplots}
\usepackage{framed}
\usepackage{graphicx}
\usepackage{subfigure}

\usepackage{algorithmic}
\usepackage{algorithm}

\usepackage{cite}

\usepackage{tabularx}
\usepackage[colorlinks,citecolor=blue]{hyperref}

\usepackage[addedmarkup=bf]{changes}
\usepackage[xcolor,pdftex]{changebar}

\usepackage{listings}

\cbcolor{red}

\setauthormarkup{}
\definechangesauthor[color=blue]{VK}
\definechangesauthor[color=purple]{DW}

\newtheorem{theorem}{Theorem}[section]
\newtheorem{corollary}{Corollary}[theorem]
\newtheorem{lemma}[theorem]{Lemma}
\newtheorem{assumption}{Assumption}
\newtheorem{definition}{Definition}[section]
\newtheorem{remark}{Remark}

\newcommand{\Uad}{U_{\text{ad}}}
\newcommand{\Yad}{Y_{\text{ad}}}
\newcommand{\Fad}{F_{\text{ad}}}
\newcommand{\Yset}{\mathcal{Y}}
\newcommand{\Aa}{\mathcal{A}^{a}}
\newcommand{\Ab}{\mathcal{A}^{b}}

\newcommand{\N}{\mathbb{N}}
\newcommand{\R}{\mathbb{R}}
\newcommand{\MO}{{\mathcal{M}(\bar{\Omega})}}
\newcommand{\CO}{{C(\bar{\Omega})}}
\newcommand{\Y}{{H^1(\Omega)\cap \CO}}
\newcommand{\U}{{L^2(\Omega)}}

\newcommand{\by}{\bar{y}}
\newcommand{\bu}{\bar{u}}
\newcommand{\bp}{\bar{p}}

\newcommand{\bmu}{\bar{\mu}}

\newcommand{\probALk}{$(P_{AL}^k)$\xspace}

\newcommand\norm[1]{\left\lVert#1\right\rVert}
\newcommand\dx{\,\mathrm{d}x}

\pgfplotsset{width=7cm,compat=1.3}

\title{A Lagrange Multiplier Method for Semilinear Elliptic State Constrained Optimal Control Problems	\thanks{This research was supported by the German Research Foundation (DFG) within the priority program "Non-smooth and
        Complementarity-based Distributed Parameter Systems: Simulation and Hierarchical
        Optimization" (SPP 1962) under
		grant number WA 3626/3-1 and NE 1941/1-1.}}
		
\author{Veronika Karl\footnote{Institut für Mathematik, Universität Würzburg, 97074 Würzburg  \newline
        V. Karl: \texttt{veronika.karl@mathematik.uni-wuerzburg.de}\newline
        I. Neitzel: \texttt{neitzel@ins.uni-bonn.de}}
        \and
        Ira Neitzel$^*$
        }
\author{
	Veronika Karl%
	\thanks{Universität Würzburg, Institut für Mathematik,
		Emil-Fischer-Str.\ 30, 97074 Würzburg, Germany; \{veronika.karl,daniel.wachsmuth\}@mathematik.uni-wuerzburg.de}
	\and
	Ira Neitzel%
	\thanks{Rheinische Friedrich-Wilhelms-Universität Bonn, Institut für Numerische Simulation, Wegelerstr.
      6, 53115 Bonn, Germany; neitzel@ins.uni-bonn.de}
       \and
       Daniel Wachsmuth$^\dagger$
}

\begin{document}
\maketitle
\begin{abstract}
\noindent In this paper we apply an augmented Lagrange method to a class of semilinear elliptic optimal control problems with pointwise state constraints. We show strong convergence of subsequences of the primal variables to a local solution of the original problem as well as weak convergence of the adjoint states and weak* convergence of the multipliers associated to the state constraint. Moreover, we show existence of stationary points in arbitrary small neighborhoods of local solutions of the original problem. Additionally, various numerical results are presented.
\end{abstract}

{\small
}%
\section{Introduction}
In this paper, the solution of an optimal control problem subject to a semilinear elliptic state equation and pointwise control and state constraints will be studied. The control problem is non-convex due to the nonlinearity of the state equation. The problem under consideration is given by
\begin{align}
\min\ J(y,u):=\frac{1}{2}||y-y_d||_{L^2(\Omega)}^2+\frac{\alpha}{2}||u||_{L^2(\Omega)}^2
\label{eq:ALS:optcontprob}\tag{$P$}
\end{align}
subject to
\begin{alignat}{2}
Ay + d(y) &=u &\quad& \text{ in }\Omega,\label{eq:ALS:pdeprob}\\
\partial_{\nu_A} y&= 0 && \text{ on }\Gamma,\notag\\
        y &\leq \psi &&\text{ in } \Omega,\notag\\
        u&\in \Uad.\notag
\end{alignat}
Here, $A$ denotes a second-order elliptic operator while $d(y)$ is a nonlinear term in $y$.
The setting of the optimal control problem will be made precise in Section \ref{sec:problem}.
\smallskip\newline
Optimal control problems with pointwise state constraints suffer from low regularity of the respective Lagrange multipliers, see \cite{Casas1986contEllProbStateConst,CasasReyesTroeltzsch2008soscSemilinState} for Dirichlet problems and \cite{Casas1993boundContSemilinState} for Neumann problems. The multiplier $\bmu$ associated to the state constraint is a Borel measure. Under additional assumptions it has been proven in \cite{CasasMateosVexler2014newRegOptContStateConst} that the multiplier satisfies $H^{-1}(\Omega)$-regularity. These assumptions are satisfied, e.g., for $\psi$ constant.

For linear quadratic optimal control problems the literature is quite rich. Quite a number of different regularization approaches have been investigated to overcome the problems that occur when solving problems of this type. We want to mention here penalization-based approaches   \cite{ItoKunisch2003semismoothState,HintermuKunisch2006pathFollowing, HintermHinzw2009moreauYosidaStateConst, HintermKunisch2009optConstStateDerivative, HintermSchielaWollner2014MYpath} and interior point methods \cite{Schiela2013intPointStateConst,KruseUlbrich2015interiorPoint}. It is a common way to reach higher regularity of the Lagrange multiplier by replacing the pure state constraints by mixed control-state-constraints as it has been done by applying Lavrentiev-regularization \cite{MeyerRoeschTroeltzsch2006optContRegState,HinzeMeyer2010lavrentiev}, or the virtual control approach. This approach has been introduced by Krumbiegel and Rösch in \cite{KrumbiegelRoesch2009virtualState} for boundary control problems. In \cite{CheredKrumbiegRoesch2008errorEstLavrentievEllipticOptCont} the approach has been adapted to linear elliptic distributed control problems and extended in \cite{KrumbiegelNeitzelRoesch2012regSemilinOptCont} to distributed elliptic optimal control problems governed by a semilinear state equation.

In lots of these approaches, the state constraints are relaxed in a suitable way, but not removed completely from the set of explicit constraints. Differently, by applying augmented Lagrange methods the state constraints are replaced by a penalized term augmenting the inequality constraint in the cost functional \cite{KanzowSteckWachsmuth16augmented,KarlWachsmuth2017augmented}. In our recent work \cite{KarlWachsmuth2017augmented} an adapted augmented Lagrange method has been analyzed in the general setting of linear elliptic optimal control problems with state constraints. Here, the presented algorithm solves sub-problems that are control constrained only. Compared to the unregularized problem the occurring sub-problems can be solved by efficient optimization algorithms. Establishing a special update rule that performs the classical augmented Lagrange update only if a sufficient decrease of the maximal constraint violation and the violation of the complementarity condition is achieved allowed us to guarantee the $L^1$-boundedness of generated multiplier approximations. The goal of the present paper is to extend this work to a larger class of optimal control problems in order to solve non-convex elliptic problems. Non-convexity arises from a semilinear state equation yielding a nonlinear solution operator.\smallskip\newline
In every iteration of the augmented Lagrange algorithm one has to solve the following sub-problem
\begin{align*}
\underset{y,u} \min &\ J(y,u) +\frac{1}{2\rho}\norm{(\mu+\rho(y-\psi))_+}^2_{L^2(\Omega)}\\
\text{s.t.}\quad &y=S(u) \qquad\text{and } \qquad u\in \Uad,
\end{align*}
where $\mu$ is a function given in $L^2(\Omega)$ and $S$ denotes the solution operator of the semilinear partial differential equation given in \eqref{eq:ALS:pdeprob}. The convergence analysis of solution algorithms of nonconvex optimal control problems suffers from non-uniqueness of local and global solutions. Due to the nonlinearity of the state equation uniqueness of the optimal solution can not be expected for the unregularized as well as for the augmented Lagrange sub-problem.
In addition, it is generally possible that critical points, which are no local solutions, are computed.
The sub-problem may have stationary points arbitrarily far from a given local solution $\bu$, and there is no rule to determine which
of these points has to be chosen in the solution process of the sub-problem in order to guarantee convergence.
That is why our first main result (Theorem \ref{theo:convKKT}) states (global) subsequential convergence towards KKT points.

We are able to extend this result under certain second-order conditions in Section \ref{sec:SSC_conv}. We will prove that the sequence that is generated by our algorithm converges to a local solution of the original problem. Furthermore, we will derive second-order optimality conditions for the sub-problem that allow us to derive local uniqueness of stationary points of the arising sub-problem.

In computations, one often uses the previous iterate $\bu_k$ as initial guess for the computation of the next iterate
$\bu_{k+1}$.
Hence, it is reasonable to expect that, if $\bu_k$ is near a local solution $\bu$
the remaining iterates will stay near $\bu$, too.
However, in this case one has to provide existence of a KKT point of the sub-problem in exactly this neighborhood.
Under a quadratic growth condition we are able to prove that for every fixed $\mu$ there exists a KKT point of the augmented Lagrange sub-problem near a local solution $\bu$, provided that the penalty parameter $\rho$ is large enough. Therefore, we investigate in Section \ref{sec:con_local_sol} the auxiliary problem
\begin{align*}
\underset{y,u}\min &\ J(y,u) +\frac{1}{2\rho}\norm{(\mu+\rho(y-\psi))_+}^2_{L^2(\Omega)}\\
\text{s.t.}\quad &y=S(u) \qquad\text{and } \qquad u\in \Uad,\qquad \norm{\bu-u}_\U\leq r
\end{align*}
that claims solutions that are close enough to a local solution $\bu$ of \eqref{eq:ALS:optcontprob}. We will prove that for $\rho$ large enough global solutions of this auxiliary problem are local solutions of the augmented Lagrange sub-problem and that these solutions converge to a local solution of the unregularized problem as %the number of iterations of the augmented Lagrange algorithm tends to infinity.\\
the penalty parameter $\rho$ tends to infinity.\smallskip\newline
The outline of this paper is as follows: In Section \ref{sec:problem} we start collecting results about the unregularized optimal control problem. Next, in Section \ref{sec:AL} we present the augmented Lagrange method. Section \ref{sec:convKKT} is dedicated to show that every weak limit point of the sequence generated by our algorithm is a KKT point of the original problem. Further, in Section \ref{sec:con_local_sol} we construct an auxiliary problem that claims solutions near a local solution of the original problem. Exploiting appropriate properties of this auxiliary problem we prove that for $\rho$ sufficiently large solutions of the auxiliary problem are local solutions of the augmented Lagrange sub-problem. Further we show convergence rates for the arising sub-problems.
In Section \ref{sec:SSC_conv} we consider second-order sufficient conditions.
To illustrate our theoretical findings we present numerical examples in Section \ref{sec:NumExp}.
\paragraph{Notation.}
Throughout the article we will use the following notation. The inner product in $L^2(\Omega)$ is denoted by $(\cdot,\cdot)$.
Duality pairings will be denoted by $\langle\cdot,\cdot\rangle$. The dual of $C(\bar{\Omega})$ is denoted by $\mathcal{M}(\bar{\Omega})$, which is the space of regular Borel measures on $\bar{\Omega}$.  Further $(\cdot)_+:=\max(0,\cdot)$ the pointwise almost-everywhere sense. We refer to $u^*$ as a (weak) limit point of a sequence $(u_k)_k$ if there exists a subsequence $(u_{k'})_{k'}$ such that $u_{k'}\rightharpoonup u^*$. If $u^*$ is the (weak) limit of $(u_k)_k$, then the whole sequence converges weakly.

\section{The Optimal Control Problem}\label{sec:problem}

Let $Y$ denote the space $Y:= H^1(\Omega)\cap \CO $, and set  $U:=\U$. We want to solve the following state constrained optimal control problem:
Minimize
\[
J(y,u):=\frac{1}{2}||y-y_d||_{L^2(\Omega)}^2+\frac{\alpha}{2}||u||_{L^2(\Omega)}^2
\]
over all $(y,u)\in Y\times \Uad$
subject to the semilinear elliptic equation
\begin{alignat*}{2}
        (Ay)(x) +d(x,y)&= u(x) &\quad& \text{in } \Omega, \\
        (\partial_{\nu_A}y)(x) &=0 && \text{on } \Gamma,
\end{alignat*}
and subject to the pointwise constraints
\begin{alignat*}{2}
        y(x)\leq \psi(x) &\quad& \text{in } \Omega,\\
        u_a(x)\leq u(x)\leq u_b(x)&&\text{in } \Omega.
\end{alignat*}
In the sequel, we will work with the following set of standing assumptions.
\begin{assumption}[\bf{Standing assumptions}]
 \begin{enumerate}
 \item Let $\Omega\subset \mathbb{R}^N$, $N=\lbrace 2,3\rbrace$ be a boun\-ded domain with $C^{1,1}$-boundary $\Gamma$ or a bounded, convex domain with polygonal boundary $\Gamma$.
%\red{ \item The set $\Uad\subseteq U$ is nonempty, convex and closed. } 
  \item The given data satisfy $y_d\in L^2(\Omega)$, 
  %$\alpha>0$, $u_a,u_b\in L^\infty(\Omega)$ with $u_a\le u_b$, 
  $\psi\in C(\bar \Omega)$.
  \item The differential operator $A$ is given by
  \[
   (Ay)(x) := -\sum_{i,j=1}^N \partial_{x_j}(a_{ij}(x)\partial_{x_i} y(x)) + a_0(x)y(x)
  \]
  with $a_{i,j} \in C^{0,1}(\bar\Omega)$ and $a_0\in L^\infty(\Omega)$. Further, $a_0\geq $ a.e. $x\in\Omega$ and $a_0\neq0$.
  The operator $A$ is assumed to satisfy the following ellipticity condition: There is $\delta>0$ such that
  \[
   \sum_{i,j=1}^N a_{ij}(x)\xi_i\xi_j\ge \delta |\xi|^2 \quad \forall \xi\in \R^N, \text{ a.e. on } \Omega.
  \]
  \item The co-normal derivative $\partial_{\nu_A}y$ is given by
  \[
  \partial_{\nu_A}y=\sum_{i,j=1}^N a_{ij}(x)\partial_{x_i}y(x)\nu_j(x),
  \]
  where $\nu$ denotes the outward unit normal vector on $\Gamma$.
  \item The function $d(x,y)\colon \Omega\times\R$ is measurable with respect to $x\in \Omega$ for all fixed $y\in \R$ and twice continuously differentiable with respect to $y$ for almost all $x\in\Omega$. Moreover, for $y=0$ the function $d$ and its derivative with respect to $y$ up to order two are bounded, i.e. there exists $C>0$ such that
  \[
  \norm{d(\cdot,0)}_\infty+\norm{\frac{\partial d}{\partial y}(\cdot,0)}_\infty+\norm{\frac{\partial^2d}{\partial y^2}(\cdot,0)}_\infty\leq C
  \]
  is satisfied. Further
  \[d_y(x,y)\geq 0 \quad \text{ for almost all } x\in \Omega.\]
  The derivatives of $d$ with respect to $y$ are uniformly Lipschitz up to order two on bounded sets, i.e, there exists a constant $M$ and a constant $L(M)$, that is dependent of $M$ such that
  \[
  \norm{\frac{\partial^2d}{\partial y^2}(\cdot,y_1)-\frac{\partial^2d}{\partial y^2}(\cdot,y_2)}_\infty \leq L(M)|y_1-y_2|
  \]
  for almost every $x\in \Omega$ and all $y_1,y_2 \in [-M,M]$.\\
  Finally, there is a subset $E_\Omega\subset \Omega$ of positive measure with $d_y(x,y)>0$ in $E_\Omega\times \R$.
 \end{enumerate}
 \label{ass:standing}
\end{assumption}

\subsection{Analysis of the Optimal Control Problem }
\subsubsection{The State Equation}
A function $y\in H^1(\Omega)$ is called a weak solution of the state equation \eqref{eq:ALS:pdeprob} if it holds for all $v\in H^1(\Omega)$ 
\[
 \int_\Omega \sum_{i,j=1}^N a_{ij}(x)\partial_{x_i}y(x)\partial_{x_j} v(x) +a_0(x)y(x)\dx  +  \int_\Omega d(x,y) v(x) \dx = \int_\Omega u(x) v(x) \dx.
\]
\begin{theorem}[\textbf{Existence of solution of the state equation}]\label{theo:stateeq}
Let Assumption \ref{ass:standing} be satisfied. Then, for every $u \in L^2(\Omega)$,
the elliptic partial differential equation
\begin{equation}\label{eq375}
\begin{alignedat}{2}
Ay + d(y)&=u &\quad& \text{in }\Omega,\\
\partial_{\nu_A} y&= 0 && \text{on }\Gamma
\end{alignedat}
\end{equation}
admits a unique weak solution $y\in H^1(\Omega)\cap \CO$, and it holds
\begin{align*}
\norm{y}_{H^1(\Omega)}+\norm{y}_{\CO}\leq c\norm{u}_{L^2(\Omega)}
%\label{eq:S-op-continuous}
\end{align*}
with $c>0$ independent of $u$. If in addition $(u_n)_n$ is such that $u_n\rightharpoonup u\in L^2(\Omega)$ then the corresponding solutions $(y_n)_n$ of \eqref{eq375} converge strongly in $\Y$ to the solution $y$ of \eqref{eq375} to data $u$.
\end{theorem}
\begin{proof}
The proof stating existence of a solution, its uniqueness, and the estimates of the norm can be found in \cite[Theorem 3.1]{Casas1993boundContSemilinState}.
The compact inclusion $L^2(\Omega)\subset H^{-1}(\Omega)$ and the fact that $u\in H^{-1}(\Omega)$ provides solutions in $\Y$ imply the additional statement.
\end{proof}
We introduce the control-to-state operator
$$ S\colon\U\rightarrow \Y,\quad u\mapsto y.$$
It is well known \cite[Theorem 4.16]{Troeltzsch2010optimal} that $S$ is locally Lipschitz continuous from $L^2(\Omega)$ to $\Y$, i.e., there exists a constant $L$ such that
\begin{align}\label{eq:S-Lipschitz}
\norm{y_1-y_2}_{H^1(\Omega)}+\norm{y_1-y_2}_{\CO}\leq L\norm{u_1-u_2}_{L^2(\Omega)}
\end{align}
is satisfied for all $u_i\in L^2(\Omega)$, $i=1,2$ with corresponding states $y_i=S(u_i)$.
We define the following sets
\begin{alignat*}{2}
\Uad&=\lbrace u\in L^\infty(\Omega)\ &&|\ u_a(x)\leq u(x)\leq u_b(x) \text{ a.e. in } \Omega\rbrace,\\
\Yad&=\lbrace y\in C(\bar{\Omega})\ &&|\ y(x)\leq \psi(x)\ \forall x\in \Omega\rbrace.
\end{alignat*}
The feasible set of the optimal control problem is denoted by
\[
\Fad = \lbrace (y,u)\in Y\times U\ |\ (y,u)\in \Yad\times \Uad,\ y=S(u)\rbrace .
\]

Using this notation the reduced formulation of problem \eqref{eq:ALS:optcontprob} is given by
\begin{align}
\underset{u\in U}\min\  f(u):=J(S(u),u),\quad s.t.\quad (S(u),u)\in \Fad.
\label{eq:optprob-gen}
\end{align}
For further use we want to recall a result concerning differentiability of the nonlinear control-to-state mapping $S$.
\begin{theorem}[\textbf{Differentiability of the solution mapping}]\label{theo:stateeq_diff}
Let Assumption \ref{ass:standing} be satisfied. Then, the mapping $S\colon L^2(\Omega) \rightarrow\Y$, that is defined by $S(u)=y$ is twice continuously Fr\'echet differentiable. Furthermore for all $u,h\in L^2(\Omega)$, $y_h=S'(u)h$ is defined as solution of
\begin{equation*}
\begin{alignedat}{2}
Ay_h+d_y(y)y_h &= h &\quad &\text{in }\Omega,\\
\partial_{\nu_A}y_h &= 0 &&\text{on } \Gamma.
\end{alignedat}
%\label{eq:LinState1}
\end{equation*}
Moreover, for every $h_1,h_2\in L^2(\Omega),\ y_{h_1,h_2}= S''(u)[h_1,h_2]$ is the solution of
\begin{equation*}
\begin{alignedat}{2}
Ay_{h_1,h_2}+d_y(y)y_{h_1,h_2} &= -d_{yy}(y)y_{h_1}y_{h_2} &\quad &\text{in }\Omega,\\
\partial_{\nu_A}y_{h_1,h_2} &= 0 &&\text{on } \Gamma,
\end{alignedat}
%\label{eq:LinState2}
\end{equation*}
where $y_{h_i}=S'(u)h_i, i=1,2$.
\end{theorem}
\begin{proof}
%For a detailed proof of the differentiability of the mapping $S:L^\infty(\Omega)\rightarrow H^1(\Omega)$ we refer to \cite[Theorem 2.5]{CasasMateos2002soscSemiLinFiniteConstraints}.
The proof for the first derivative of $S \colon L^r(\Omega)\rightarrow \Y, r>N/2$ can be found in \cite[Theorem 4.17]{Troeltzsch2010optimal}. We refer to \cite[Theorem 4.24]{Troeltzsch2010optimal} for the proof of second-order differentiability of $S\colon L^\infty(\Omega)\rightarrow \Y$ which is also valid for $S\colon L^2(\Omega)\rightarrow \Y$.
\end{proof}

\subsubsection{Existence of Solutions of the Optimal Control Problem}
Under the standing assumptions we can show existence of solutions of the reduced control problem \eqref{eq:optprob-gen}.
\begin{definition}[\textbf{Local solution}]%\label{def:local-sol}
A control $\bu\in \Uad$ satisfying $S(\bu) \leq \psi$ in $\bar{\Omega}$ is called a local solution of problem \eqref{eq:ALS:optcontprob} in the sense of $L^2(\Omega)$ if there exists a $\zeta>0$ such that
\begin{align*}
f(\bu)\leq f(u) \quad \text{for all } u\in\Uad \text{ with } S(u)\leq \psi \text{ in }\bar{\Omega} \text{ and } \norm{\bu-u}_{L^2(\Omega)}\leq \zeta.
\end{align*}
\end{definition}
By standard arguments we get the following theorem.
\begin{theorem}[\textbf{Existence of solution of the optimal control problem}]
Let Assumption \ref{ass:standing} be satisfied. Assume that the feasible set $\Fad$ is nonempty. Then, there exists at least one global solution $(\by,\bu)$ of \eqref{eq:ALS:optcontprob}.
%\label{theo:exsolutionProb}
\end{theorem}
\begin{proof}
The proof can be found in \cite[Theorem 1.45]{HinzePinnauUlbrich2009optimization}.
\end{proof}
Due to non-convexity, global solutions of problem \eqref{eq:ALS:optcontprob}
are not unique in general, also, in addition there
might be local solutions.
\subsubsection{First-Order Optimality Conditions}
The existence of Lagrange multipliers to state constrained optimal control problems is not guaranteed without some regularity assumption.
In order to formulate first-order necessary optimality conditions we will work with the following linearized Slater condition.
\begin{assumption}[\textbf{Linearized Slater condition}]
\label{ass:slater}
We assume that a local solution $\bu$ satisfies the linearized Slater condition, i.e., there exists $\hat{u}\in \Uad$
and $\sigma>0$ such that there holds
\[
S(\bu)(x)+S'(\bu)(\hat{u}-\bu)(x) \leq \psi(x)-\sigma \quad \forall x\in \bar{\Omega}.
\]
\end{assumption}
Next, we
state a regularity result concerning linear partial differential equations with measure on the right-hand side, see \cite[Theorem 4.3]{Casas1993boundContSemilinState}.
\begin{theorem}[\textbf{Existence of solution of the adjoint equation}]\label{theo:exsoladjointeq}
Let $\mu$ be a regular Borel measure with $\mu=\mu_\Omega+\mu_\Gamma\in \MO$. Then the elliptic partial differential equation
\begin{equation*}
\begin{alignedat}{2}
A^*p + d_y(y)p&=y-y_d +\mu_\Omega &\quad& \text{in }\Omega,\\
\partial_{\nu_{A^*}} p&= \mu_\Gamma && \text{on }\Gamma
\end{alignedat}
\end{equation*}
admits a unique weak solution $p\in W^{1,s}(\Omega),s\in[1,N/(N-1))$ and it holds
\begin{align*}
\norm{p}_{W^{1,s}(\Omega)}\leq c\left(\norm{y}_{L^2(\Omega)}+\norm{y_d}_{L^2(\Omega)}+\norm{\mu}_\MO\right)
\end{align*}
with $c>0$ independent of the right hand side of the partial differential equation.
\end{theorem}
Based on the linearized Slater condition first-order necessary optimality conditions for problem \eqref{eq:ALS:optcontprob} can be established.
\begin{theorem}[\textbf{First-order necessary optimality conditions}]\label{theo:KKT-AL}
%\label{theo:ex-adjoint-multiplier}
Let $\bu$ be a local solution of problem \eqref{eq:ALS:optcontprob} that satisfies Assumption \ref{ass:slater}. Let $\by=S(\bu)$ denote the corresponding state. Then, there exists an adjoint state $\bp\in W^{1,s}(\Omega)$, $s\in(1,N/(N-1))$ and a Lagrange multiplier $\bar{\mu}\in \MO$ with $\bmu=\bmu_{\Omega}+\bmu_{\Gamma}$
such that the following optimality system
\begin{subequations}\label{kkt_optsys}
\begin{equation}
\begin{alignedat}{2}
A\by +d(\by)&=\bu  &\quad &\text{in } \Omega,\\
\partial_{\nu_A}\by&=0 &&\text{on } \Gamma,
\end{alignedat}
\label{eq:kkt_o:1}
\end{equation}
\begin{equation}
\begin{alignedat}{2}
A^*\bp +d_y(\by)\bp &= \by-y_d +\bar{\mu}_\Omega &\quad &\text{in } \Omega,\\
\partial_{\nu_{A^*}}\bp&=\bar{\mu}_\Gamma &&\text{on } \Gamma,
\end{alignedat}
\label{eq:kkt_o:2}
\end{equation}
\begin{equation}
( \bp+\alpha \bu,u-\bu) \geq 0\quad\forall u\in \Uad, \label{eq:kkt_o:3}
\end{equation}
\begin{equation}
\langle\bar{\mu},\by-\psi\rangle_{\MO,\CO}=0,\quad  \bar{\mu}\geq 0,\quad \by(x)\leq \psi(x),\ \forall x\in\bar{\Omega}\label{eq:kkt_o:4}
\end{equation}
\label{eq:kkt_o}
\end{subequations}
is fulfilled. Here, the inequality $\bar{\mu}\geq 0$ means $\langle \bar{\mu},\varphi\rangle_{\MO,\CO}\geq 0$ for all $\varphi\in \CO$ with $\varphi\ge0$.
\end{theorem}
\begin{proof}
The proof can be done by adapting the theory from \cite[Theorem 5.3]{Casas1993boundContSemilinState} to Neumann boundary conditions.
\end{proof}
Let us emphasize that due to the presence of control as well as state constraints, the adjoint state $\bp$ and the Lagrange multiplier $\bmu$ need not to be unique. 

\section{The Augmented Lagrange Method}\label{sec:AL}
Like in \cite{KarlWachsmuth2017augmented} we eliminate the explicit state constraint $S(u)\leq\psi$ from the set of constraints by adding an augmented Lagrange term to the cost functional. Let $\rho>0$ denote a penalization parameter and $\mu$ a fixed function in $L^2(\Omega)$. Then in every step $k$ of the augmented Lagrange method one has to solve the sub-problem
\begin{align}
\underset{u_{\rho}}{\min}\ &f_{AL}(u_{\rho},\mu,\rho):=f(u_\rho)+\frac{1}{2\rho}\int_{\Omega}\left(\left(\mu+\rho(S(u_{\rho})-\psi)\right)_+\right)^2 \dx
\label{prob_auglag}\tag{${P_{AL}^{\rho,\mu}}$}
\end{align}
where $(\cdot)_+:=\max(0,\cdot)$ in the pointwise sense, subject to the  control constraints
$$ u_{\rho}\in \Uad.$$
\subsection{Analysis of the Augmented Lagrange Sub-Problem}
In the following, existence of an optimal control and existence of a corresponding adjoint state will be proven. Local solutions of the augmented Lagrange sub-problem \eqref{prob_auglag} are defined analogously to \eqref{eq:ALS:optcontprob}.

\begin{definition}[\bf{Local solution}]
A control $\bu_\rho\in \Uad$ is a local solution of the augmented Lagrange sub-problem \eqref{prob_auglag} if there exists a $\zeta>0$ such that 
\begin{align*}
 f_{AL}(\bu_\rho) \leq f_{AL} (u),\quad \text{ for all } u\in \Uad \text{ with } \norm{u-\bu_\rho}_{L^2(\Omega)}\leq \zeta.
\end{align*}
\end{definition}

\begin{theorem}[\textbf{Existence of solutions of the augmented Lagrange sub-problem}]
For every $\rho>0$, $\mu\in L^2(\Omega)$ with $\mu\geq 0$ the augmented Lagrange sub-problem \eqref{prob_auglag} admits at least one global solution $\bu_{\rho}\in \Uad$.% with associated optimal state $\by_{\rho}\in \Y$.
\end{theorem}
\begin{proof}
The proof follows standard arguments, see \cite{Troeltzsch2010optimal}.
\end{proof}
Since the problem \eqref{prob_auglag} has no state constraints, the first-order optimality system is fulfilled without any further regularity assumptions.

\begin{theorem}[\textbf{First-order necessary optimality conditions}]%\label{thm_optcon_aug}
For given $\rho>0$ and $0\leq\mu\in L^2(\Omega)$ let $(\by_{\rho},\bu_{\rho})$ be a solution of \eqref{prob_auglag}. Then, there exists a unique adjoint state $\bp_{\rho}\in H^1(\Omega)$ satisfying the following system
\begin{subequations}\label{AL:optsys}
\begin{equation}\label{AL1_pde}
\begin{alignedat}{2}
A\by_{\rho}+d(\by_{\rho})&=\bu_{\rho} &\quad &\text{in } \Omega,\\
\partial_{\nu_A} \by_{\rho}&=0 &&\text{on } \Gamma,
\end{alignedat}
\end{equation}
\begin{equation} \label{AL2_adjoint}
\begin{alignedat}{2}
A^*\bar{p}_{\rho} + d_y(\by_{\rho})\bp_{\rho}&= \bar{y}_{\rho}-y_d +\bmu_{\rho} &\quad &\text{in } \Omega,\\
\partial_{\nu_{A^*}}\by_{\rho}&=0 &&\text{on } \Gamma,
\end{alignedat}
\end{equation}
\begin{equation} \label{AL3_varineq}
(\bar{p}_{\rho} +\alpha \bu_{\rho},u-\bu_{\rho})\geq 0,\qquad \forall u\in \Uad
\end{equation}
\begin{equation} \label{AL4_mult}
\bmu_{\rho}=\left(\mu+\rho(\by_{\rho}-\psi)\right)_+.
\end{equation}
\end{subequations}
\end{theorem}
\begin{proof}
For the existence of an adjoint state $\bp_\rho\in H^1(\Omega)$ that satisfies the KKT system we refer to \cite[Corollary 1.3, p.73]{HinzePinnauUlbrich2009optimization}. By construction  we get a unique $\bmu_\rho$ for each $(\by_\rho,\bu_\rho)$. Due to Theorem \ref{theo:exsoladjointeq} the adjoint equation admits a unique solution. Thus, the adjoint state $\bp_\rho$ is unique.
\end{proof}
Finally, in Algorithm \ref{alg_detail} we present the augmented Lagrange algorithm, which is based on the algorithm that has been developed in \cite{KarlWachsmuth2017augmented}.

\begin{algorithm}[H]\caption{Augmented Lagrange Algorithm}\label{alg_detail}
Let $\rho_1>0$ and $\mu_1\in L^2(\Omega)$ be given with $\mu_1\ge0$. Choose $\theta>1$, $\tau\in(0,1)$, $\epsilon\ge0$, $R_0^+ \gg 1$. Set $k:=1$ and $n:=1$.
\begin{enumerate}
 \item Solve the optimality system \eqref{AL:optsys} for $\mu:=\mu_k$, and obtain $(\by_{k},\bu_{k},\bp_{k})$.
 \item Set $\bmu_k:=(\mu_k+\rho_k(\by_k-\psi))_+$.
 \item  Compute $R_k:= \norm{(\by_{k}-\psi)_+}_\CO + ( \bmu_{k},\psi-\by_{k})_+$.
 \item If $R_k\le \tau R^+_{n-1}$ then the step $k$ is successful, set
 $$\mu_{k+1}:=\bmu_k=\left(\mu_k+\rho_k(\by_k-\psi)\right)_+,$$
 $\rho_{k+1}:=\rho_k$, and define $(y_n^+,u_n^+,p_n^+):=(\by_{k},\bu_{k},\bp_{k})$, as well as $\mu_n^+:=\mu_{k+1}$ and $R_n^+:=R_k$.
 Set $n:=n+1$.
 \item Otherwise the step $k$ is not successful, set $\mu_{k+1}:=\mu_k$, increase penalty parameter $\rho_{k+1}:=\theta \rho_k$.
 \item If $R_{n-1}^+ \le \epsilon$ then stop, otherwise set $k:=k+1$ and go to step 1.
\end{enumerate}
\end{algorithm}
\noindent
In the following we will call the step $k$ \emph{successful} if the quantity
\[
R_k:= \norm{(\by_{k}-\psi)_+}_\CO + ( \bmu_{k},\psi-\by_{k})_+
\]
shows sufficient decrease (see step 4 of the algorithm). Otherwise we will call the step \emph{not successful}. 
The first part of $R_k$ measures the maximal constraint violation while the second term quantifies the fulfilment of the complementarity condition in the second part. Since $(\bmu_{k}(x),\psi(x)-\by_{k}(x))$ is nonnegative for every feasible $\by_k$ it is enough to check on the smallness of $( \bmu_{k},\psi-\by_{k})_+$ in the second term for quantifying if the complementarity condition is satisfied.\medskip\\
From now on let $(P_{AL}^k)$ denote the augmented Lagrange sub-problem \eqref{prob_auglag}
for given penalty parameter $\rho:=\rho_k$ and multiplier $\mu:=\mu_k$. We will denote its solution by $(\by_k,\bu_k)$ with adjoint state $\bp_k$ and updated multiplier $\bmu_k$.

\section{Convergence Analysis}\label{sec:convKKT}

\subsection{Infinitely Many Successful Steps and Convergence Towards Feasible Points} 
The most crucial part of the convergence analysis is to prove that the algorithm makes infinitely many successful steps. Otherwise the algorithm might be caught in an infinite loop between the steps $1,2,3$ and $5$. \medskip\newline
The following assumption plays the key role for proving that Algorithm \ref{alg_detail} is well-defined. 

\begin{assumption}\label{ass:termMuBounded}
In step 1 of Algorithm \ref{alg_detail}, the solutions $(\by_k,\bu_k,\bp_k)$ of \eqref{AL:optsys}
are chosen such that
\[
\frac{1}{\rho_k}\norm{\bmu_k}^2_{L^2(\Omega)}=\frac{1}{\rho_k}\norm{(\mu_k+\rho_k(\by_k-\psi))_+}^2_{L^2(\Omega)}
\]
is uniformly bounded.
\end{assumption}

Working with Assumption \ref{ass:termMuBounded} we consider different approaches.
In the first approach we consider the sequence $(\bu_k)_k$ generated by Algorithm \ref{alg_detail} which is, due to the control constraints, bounded in $L^2(\Omega)$. Hence, we can extract a weakly converging subsequence $\bu_{k'}\rightharpoonup u^*$ in $L^2(\Omega)$. Note, that here $u^*$ denotes only a weak limit point of $(\bu_{k'})_{k'}$ and \emph{not} necessarily a local solution of the optimal control problem \eqref{eq:ALS:optcontprob}. Further, by Theorem \ref{theo:stateeq} we get a strongly converging subsequence $\by_{k'}\rightarrow y^*$ in $H^1(\Omega)\cap \CO$. 
Exploiting Assumption \ref{ass:termMuBounded} our aim is to show that $y^*:=S(u^*)$ is feasible which in turn will yield that the term $R_k$ tends to zero (Theorem \ref{theo:infsucstep}).
In the second approach we choose $(\by_k,\bu_k)$ to be global minimizers of the augmented Lagrange sub-problem and show via a contradiction argument that infinitely many successful steps are done which in turn yields feasibility of any accumulation point of $(S(\bu_k))_k$. We start with the first approach proving an auxiliary result that does not require Assumption \ref{ass:termMuBounded}.

\begin{lemma}\label{lem:limsup_compl_mufixed}
Assume that only finitely many steps of Algorithm \ref{alg_detail} are successful.
Then it holds
$$\underset{k\rightarrow\infty}{\lim\sup}\ (\bmu_k,\psi-\by_k)_+\leq 0.$$
\end{lemma}
\begin{proof}
By assumption, there is an index $m$ such that all steps $k$ with $k>m$ are not successful.
According to Algorithm \ref{alg_detail} it holds $\mu_k=\mu_m$ for all $k>m$.
Let 
$$\Omega\supseteq\Omega_k:=\left\lbrace x\in \Omega \colon (\bmu_k,\psi-\by_k)(x)\geq 0 \right\rbrace.$$
Then, the desired estimate follows easily by pointwise evaluation
of the contributing quantities in
\begin{align*}
(\bmu_k,\psi-\by_k)_+ &= (\bmu_k,-\frac{\mu_m}{\rho_k}+\psi-\by_k+\frac{\mu_m}{\rho_k})_+\leq -\frac{1}{\rho_k}\norm{\bmu_k}^2_{L^2(\Omega_k)}+\frac{1}{\rho_k}(\bmu_k,\mu_m)_{L^2(\Omega_k)}\\
&\leq -\frac{1}{2\rho_k}\norm{\bmu_k}^2_{L^2(\Omega_k)} +\frac{1}{2\rho_k}\norm{\mu_m}^2_{L^2(\Omega_k)}\leq \frac{1}{2\rho_k}\norm{\mu_m}^2_{L^2(\Omega)},
\end{align*}
where we applied Young's inequality.
\end{proof}

We will now use Assumption \ref{ass:termMuBounded} to prove feasibility of $y^*=S(u^*)$.

\begin{lemma}\label{lemma:convy_feasible}
Let Assumption \ref{ass:termMuBounded} be satisfied. Further, let $(\mu_k)_k \in L^2(\Omega)$ and let $(\rho_k)_k$ be a sequence of positive numbers with $\rho_k\rightarrow\infty$. Let $(\by_k,\bu_k,\bp_k)_k$ be a sequence of solutions of \eqref{AL:optsys}. Let $u^*$ denote a weak limit point of $(\bu_k)_k$. Then the associated state $y^*=S(u^*)$ is feasible, i.e., $y^*\leq \psi$.
\begin{proof}
For every subsequence $\bu_{k'}\rightharpoonup u^*$ in $\U$ we have due to Theorem \ref{theo:stateeq} $\by_{k'}\rightarrow y^*$ in $\Y$  and hence $\by_{k'}\rightarrow y^*$ in $L^2(\Omega)$. Due to Assumption \ref{ass:termMuBounded} the identity
\begin{align*}
\frac{1}{\rho_{k'}}\norm{\bmu_{k'}}^2_{L^2(\Omega)} = \rho_{k'}\norm{\left(\frac{\mu_{k'}}{\rho_{k'}}+\by_{k'}-\psi\right)_+}^2_{L^2(\Omega)}
%\label{eq:proof-identity}
\end{align*}
is bounded in $L^2(\Omega)$. Exploiting $\norm{(\by_{k'}-\psi)_+}^2_{L^2(\Omega)}\leq \norm{\left(\frac{\mu_{k'}}{\rho_{k'}}+\by_{k'}-\psi\right)_+}^2_{L^2(\Omega)}$, we can argue $y^*\leq \psi$ for $k'\rightarrow\infty$.
\end{proof}
\end{lemma}

Exploiting Lemma \ref{lemma:convy_feasible} it can be shown that the augmented Lagrange algorithm makes infinitely many successful steps.
\begin{theorem}[\textbf{Infinitely many successful steps}]
Under Assumption \ref{ass:termMuBounded} the augmented Lagrange algorithm makes infinitely many successful steps.
\label{theo:infsucstep}
\end{theorem}
\begin{proof}
As in the proof of \cite[Lemma 3.8]{KarlWachsmuth2017augmented} we
assume that the algorithm does a finite number of successful steps, only. Then there is an index $m$ such that all steps $k$ with $k>m$ are not successful. According to Algorithm \ref{alg_detail} it holds $\mu_k=\mu_m$ for all $k>m$, $R_k>\tau R_m>0$ and $\rho_k\rightarrow\infty$. However, by Lemma \ref{lem:limsup_compl_mufixed} and Lemma \ref{lemma:convy_feasible} we find a subsequence
$$  \lim_{k'\rightarrow\infty}R_{k'}=\lim_{k'\rightarrow\infty}\norm{(\by_{k'}-\psi)_+}_{C(\bar{\Omega})} + ( \bmu_{k'},\psi-\by_{k'})_+=0,$$
yielding a contradiction.
\end{proof}

Let us recall that Assumption  \ref{ass:termMuBounded} which is the basis for proving that the algorithm makes infinitely many successful steps is a rather strong assumption. We therefore want to argue that it can be satisfied, if we take $(\by_k,\bu_k)$ to be global minimizers of the augmented Lagrange sub-problem.

\begin{lemma}\label{lemma:Ass_mubounded_satisfied_globSol}
Assume that in step 1 of Algorithm \ref{alg_detail}, the pair
$(\by_k,\bu_k)$ is chosen to be a global minimizer of the augmented Lagrange sub-problem.
Assume that only finitely many steps of Algorithm \ref{alg_detail} are successful.
Then Assumption \ref{ass:termMuBounded} is satisfied.
\end{lemma}
\begin{proof}
Let $(S(\bu),\bar u)$ be a global solution of the original problem. Let $k>m$, where $m$ is the largest index
of a successful step. This implies $\mu_k=\mu_m$. Then we obtain
\begin{align*}
0&\leq f_{AL}(\bu_k,\mu_k,\rho_k) = f(\bu_k) + \frac1{2\rho_k}\norm{\bmu_k}^2_{L^2(\Omega)}\\
 &\le f_{AL}(\bu,\mu_k,\rho_k)
= f(\bar u)  + \frac1{2\rho_k}\norm{(\mu_k+\rho_k(S(\bu)-\psi))_+}^2_{L^2(\Omega)}\\
&= f(\bar u)  + \frac1{2\rho_k}\norm{(\mu_m+\rho_k(S(\bu)-\psi))_+}^2_{L^2(\Omega)}\leq f(\bar u)  + \frac1{2\rho_k}\norm{\mu_m}^2_{L^2(\Omega)}.
\end{align*}
Hence, Assumption \ref{ass:termMuBounded} is clearly satisfied.
\end{proof}

%This leads us to the following Lemma.
\begin{lemma}\label{lemma:globsol_infsuc}
Assume that in step 1 of Algorithm \ref{alg_detail}, the pair
$(\by_k,\bu_k)$ is chosen to be global minimizers of the augmented Lagrange sub-problem. Then the augmented Lagrange algorithm makes infinitely many successful steps and any limit point $y^*$ of $(y_n^+)_n$ corresponding to $(u_n^+)_n$ is feasible for \eqref{eq:ALS:optcontprob}.
\end{lemma}
\begin{proof}
Assuming that only finitely many steps are successful we know from Lemma \ref{lemma:Ass_mubounded_satisfied_globSol} that Assumption \ref{ass:termMuBounded} is satisfied. However, then from Theorem \ref{theo:infsucstep} we obtain a contradiction. Hence we know that Algorithm \ref{alg_detail} makes infinitely many successful steps. Since $R_n^+$ tends to zero, the term $\norm{(y_n^+-\psi)_+}_\CO$ yields feasibility of any limit point of $(y_n^+)_n$.
\end{proof}

Without any further assumptions our algorithm yields the following convergence properties.

\begin{theorem}[\bf{Convergence to feasible points}] \label{theo:statpoint}
Let $(y_n^+,u_n^+,p_n^+,\mu_n^+)_n$ denote a sequence generated by Algorithm \ref{alg_detail}. Let Assumption \ref{ass:termMuBounded} be satisfied. Let $u^*$ denote a weak limit point of $(u_n^+)_n$.
Then $y^*=S(u^*)$ is feasible i.e., $y^*\leq\psi$ and it holds $\lim_{n'\rightarrow\infty}(\mu_{n'}^+,\psi-y_{n'}^+)_+=0$.
\end{theorem}
\begin{proof}
By the boundedness of $(u_{n}^+)_{n}\in \Uad$ we get existence of a subsequence $u_{n'}^+\rightharpoonup u^*$ in $L^2(\Omega)$ and $y_{n'}^+\rightarrow y^*$ in $\Y$. Theorem \ref{theo:infsucstep} guarantees that the algorithm makes infinitely many successful steps. Hence $$0\leq \lim_{n'\rightarrow \infty} R_{n'}^+ = \lim_{n'\rightarrow \infty} \norm{(y_{n'}^+-\psi)_+}_{\CO} + ( \mu_{n'}^+,\psi-y_{n'}^+)_+ = 0.$$ Thus $y^*\leq \psi$ and $\lim_{n'\rightarrow\infty}(\mu_{n'}^+,\psi-y_{n'}^+)_+=0$.
\end{proof}

In Theorem \ref{theo:statpoint} we have proven that a weak limit point $u^*$ of $(u_n^+)_n$ with corresponding state $y^*$ is feasible for \eqref{eq:ALS:optcontprob}. However, we do not know yet, if $u^*$ is a stationary point, i.e., if $(p_n^+,\mu_n^+)_n$ converges in some sense to $(p^*,\mu^*)$ such that $(y^*,u^*,p^*,\mu^*)$ satisfies the optimality system \eqref{eq:kkt_o}.
To achieve this aim, we have to suppose additional properties of the weak limit point $u^*$. In the next subsection we will investigate the impact on our convergence result if our algorithm generates a sequence with weak limit point $u^*$ that satisfies a linearized Slater condition.

\subsection{Convergence towards KKT Points}%\label{sec:ConvKKT}
We have shown in the previous section that the augmented Lagrange algorithm converges on a subsequence to a feasible point. Now we want to extend our results by proving convergence to a KKT point. We start with several auxiliary results.

\begin{lemma}\label{lemma:convLinearizedState}
Let $(u_k)_k,(h_k)_k$ denote sequences in $L^2(\Omega)$ that converge weakly to the limits $u^*,h^*$, respectively. 
%Further, let $v$ denote a fixed function in $L^2(\Omega)$. 
Then, for $k\rightarrow\infty$ we have
%$$\norm{S'(u_k)h_k - S'(u^*)h^*}_{\CO}\rightarrow 0  \quad\text{ and }\quad  \norm{S''(v)(h_k)^2 - S''(v)(h^*)^2}_\CO\rightarrow 0.$$
$$\norm{S'(u_k)h_k - S'(u^*)h^*}_{\CO}\rightarrow 0 .$$
\end{lemma}
\begin{proof}
From Theorem \ref{theo:stateeq} we know that $y_k:=S(u_k)$ is the unique weak solution of the state equation
\begin{alignat*}{3}
-\Delta y_k +d(y_k) &= u_k &\quad &\text{ in }\Omega,\\
\partial_{\nu_A}y_k &= 0 &&\text{ on }\Gamma.
\end{alignat*}
Further for $u_k\rightharpoonup u^*$ in $L^2(\Omega)$ we get $y_k\rightarrow y^*$ in $\Y$. Let now $z_k$ denote the linearized state $z_k:= S'(u_k)h_k$. Then by Theorem \ref{theo:stateeq_diff} we know that $z_k$ is the unique solution of
\begin{alignat*}{3}
-\Delta z_k +d_y(y_k)z_k &= h_k &\quad &\text{ in }\Omega,\\
\partial_{\nu_A}z_k &= 0 &&\text{ on }\Gamma.
\end{alignat*}
Further let $z^*:=S'(u^*)h^*$ solve the equation
\begin{alignat*}{3}
-\Delta z^* +d_y(y^*)z^* &= h^* &\quad &\text{ in }\Omega,\\
\partial_{\nu_A}z^* &= 0 &&\text{ on }\Gamma.
\end{alignat*}
We subtract both PDEs and set $e_k:=S'(u_k)h_k-S'(u^*)h^*$
\begin{alignat*}{3}
-\Delta e_k +d_y(y_k)z_k-d_y(y^*)z^* &= h_k-h^* &\quad &\text{ in }\Omega,\\
\partial_{\nu_A} e_k &= 0 &&\text{ on }\Gamma.
\end{alignat*}
Inserting the identity $d_y(y_k)z_k-d_y(y^*)z^*= \left(d_y(y_k)-d_y(y^*)\right)z_k + d_y(y^*)(z_k-z^*)$ we obtain
\begin{alignat*}{3}
-\Delta e_k+d_y(y^*)e_k &= (h_k-h^*) - (d_y(y_k)-d_y(y^*))z_k&\quad &\text{ in }\Omega,\\
\partial_{\nu_A} e_k &= 0 &&\text{ on }\Gamma.
\end{alignat*}
From Assumption \ref{ass:standing} we know that $d_y(y)$  is locally Lipschitz continuous, i.e.,
\begin{align*}
\norm{d_y(y_1)-d_y(y_2)}_{L^\infty(\Omega)}\leq L\norm{y_1-y_2}_{L^\infty(\Omega)}.
\end{align*}
Concluding, for $y_{k}\rightarrow y^*$ in $L^\infty(\Omega)$ we have $d_y(y_k)\rightarrow d_y(y^*)$ in $L^\infty(\Omega)$. Due to $h_k\rightharpoonup h^*$ in $L^2(\Omega)$ and the boundedness of $z_k$ in $L^2(\Omega)$ we gain $e_k\rightarrow 0$ in $\Y$. Hence,
$$ \norm{S'(u_k)h_k-S'(u^*)h^*}_{\CO}\rightarrow 0$$
and the proof is done.
\end{proof}

Let us recall that $(y_n^+,u_n^+,p_n^+,\mu_n^+)$ denotes the solution of the $n$-th successful iteration of Algorithm \ref{alg_detail}. We want to investigate the convergence properties of the algorithm for a weak limit point $u^*$ of $(u_n^+)_n$. A point $u^*\in\Uad$ satisfies the linearized Slater condition if there exists a $\hat{u}\in \Uad$ and $\sigma>0$ such that
\begin{align}\label{eq:Slater_limitpoint}
S(u^*)(x)+S'(u^*)(\hat{u}-u^*)(x)\leq \psi(x) -\sigma\quad \forall x\in\bar{\Omega}.
\end{align}

\begin{lemma}\label{lemma:SucStep-LinSlater}
Let $u^*$ denote a 
%feasible 
weak limit point of $(u_n^+)_n$ that satisfies the linearized Slater condition \eqref{eq:Slater_limitpoint}. Then, there exists an $N\in \N$ such that for all $n'>N$ the control $u_{n'}^+$ satisfies 
\begin{align}\label{eq:Slater_nlarge}
S(u_{n'}^+)+S'(u_{n'}^+)(\hat{u}-u_{n'}^+) \leq \psi -\frac{\sigma}{2}.
\end{align}
\end{lemma}
\begin{proof}
By Theorem \ref{theo:statpoint} we have strong convergence $S(u_{n'}^+)\rightarrow S(u^*)$ in $\Y$.
% where the limit $S(u^*)$ is feasible. 
By Theorem \ref{lemma:convLinearizedState} we get $S'(u_{n'}^+)(\hat{u}-u_{n'}^+) \rightarrow S'(u^*)(\hat{u} - u^*)$ in $\Y$.
Using the identity
\begin{align*}
S(u_{n'}^+)+S'(u_{n'}^+)(\hat{u}-u_{n'}^+) &= S(u^*)+S'(u^*)(\hat{u}-u^*)\\
& \quad + S(u_{n'}^+)-S(u^*)\\
&\quad + S'(u_{n'}^+)(\hat{u}-u_{n'}^+)-S'(u^*)(\hat{u}-u^*)
\end{align*}
and exploiting the specified convergence results, we conclude of an $N\in \mathbb{N}$ such that
\begin{align*}
S(u_{n'}^+)+S'(u_{n'}^+)(\hat{u}-u_{n'}^+) &\leq \psi - \frac{\sigma}{2},\quad \forall {n'}>N. \qedhere
\end{align*}
\end{proof}

We recall an estimate for the second term of the update rule, see \cite[Lemma 3.9]{KarlWachsmuth2017augmented}, that is necessary to state $L^1$-boundedness of the Lagrange multiplier. This estimate does not require any additional assumption, it just results from the structure of the update rule.

\begin{lemma}
\label{lemma:scaprod_alg}
Let $y_n^+,\mu_n^+$ be given as defined in Algorithm \ref{alg_detail}. Then for all $n>1$ it holds
\begin{align*}
( \mu_n^+,\psi-y_n^+)_+\leq \tau^{n-1}\left(\norm{(y_1^+-\psi)_+}_{C(\bar{\Omega})}+\norm{\mu_1^+}_{L^2(\Omega)}\norm{(\psi-y_1^+)_+}_{L^2(\Omega)}\right).
%\label{eq:est_scaprod}
\end{align*}
\end{lemma}

\begin{lemma}[\textbf{Boundedness of the Lagrange multiplier}]\label{lemma:multL1_boundedness}
Assume that Algorithm \ref{alg_detail} generates the sequence $(y_n^+,u_n^+,p_n^+,\mu_n^+)_n$. Let $(u_{n'}^+)_{n'}$ denote a subsequence of $(u_n^+)_n$ that converges weakly to $u^*$. If $u^*$ satisfies the linearized Slater condition from \eqref{eq:Slater_limitpoint}, then the corresponding sequence of multipliers $(\mu_{n'}^+)_{n'}$ is bounded in $L^1(\Omega)$, i.e., there is a constant C > 0 independent of $n'$ such that for all $n'$ it holds
$$ \norm{\mu_{n'}^+}_{L^1(\Omega)} \leq C. $$
\end{lemma}
\begin{proof}

Writing \eqref{AL3_varineq} in variational form we see
\begin{align*}
(p_{n'}^+ +\alpha u_{n'}^+,u-u_{n'}^+)\geq 0 \qquad \forall u\in \Uad.
\end{align*}
Using the identity
$$ p_{n'}^+=S'(u_{n'}^+)^*(y_{n'}^+ -y_d+\mu_{n'}^+)$$
we obtain
\begin{align*}
(S'(u_{n'}^+)^*(y_{n'}^+ -y_d+\mu_{n'}^+)+\alpha u_{n'}^+,u-u_{n'}^+)\geq 0, \qquad \forall u\in \Uad.
\end{align*}
Rearranging terms yields
\begin{align*}
(\mu_{n'}^+,S'(u_{n'}^+)(u_{n'}^+-u))\leq (y_{n'}^+-y_d, S'(u_{n'}^+)(u-u_{n'}^+))+(\alpha u_{n'}^+,u-u_{n'}^+).
\end{align*}
Testing the left hand side of the previous inequality with the test function $u:=\hat{u}\in\Uad$ we get
\begin{align*}
 (\mu_{n'}^+,S'(u_{n'}^+)(u_{n'}^+-\hat{u}))&= (\mu_{n'}^+,S'(u_{n'}^+)(u_{n'}^+-\hat{u})) +(\mu_{n'}^+,S(u_{n'}^+)-\psi) -(\mu_{n'}^+,S(u_{n'}^+)-\psi)\\
&= - (\mu_{n'}^+,S(u_{n'}^+)+S'(u_{n'}^+)(\hat{u}-u_{n'}^+)-\psi) +(\mu_{n'}^+,S(u_{n'}^+)-\psi).
\end{align*}
By Lemma \ref{lemma:SucStep-LinSlater} we know that there exists an $N$ such that for all $n'>N$ the control $u_{n'}^+$ satisfies \eqref{eq:Slater_nlarge}. Hence for all $n'>N$ we obtain
\begin{align*}
\frac{\sigma}{2}\norm{\mu_{n'}^+}_{L^1(\Omega)}\leq - (\mu_{n'}^+,S(u_{n'}^+)+S'(u_{n'}^+)(\hat{u}-u_{n'}^+)-\psi).
\end{align*}
Thus, we estimate
\begin{align*}
\frac{\sigma}{2}&\norm{\mu_{n'}^+}_{L^1(\Omega)} \leq (\mu_{n'}^+,\psi - S(u_{n'}^+))+(y_{n'}^+-y_d, S'(u_{n'}^+)(\hat{u}-u_{n'}^+))+(\alpha u_{n'}^+,\hat{u}-u_{n'}^+)\\
&\leq  (\mu_{n'}^+,\psi-y_{n'}^+)_+ +(y_{n'}^+ -y_d,S'(u_{n'}^+)(\hat{u}-u_{n'}^+)) -\frac{\alpha}{2}\norm{u_{n'}^+-\hat{u}}_{L^2(\Omega)}^2+\frac{\alpha}{2}\norm{\hat{u}}_{L^2(\Omega)}^2
\end{align*}
and hence
\begin{align*}
\frac{\sigma}{2}\norm{\mu_{n'}^+}_{L^1(\Omega)} &+\frac{\alpha}{2}\norm{u_{n'}^+-\hat{u}}_{L^2(\Omega)}^2 \\
&\leq (\mu_{n'}^+,\psi-y_{n'}^+)_+  +\norm{y_{n'}^+-y_d}_{L^2(\Omega)}\norm{S'(u_{n'}^+)(\hat{u}-u_{n'}^+)}_{L^2(\Omega)} + \frac{\alpha}{2}\norm{\hat{u}}_{L^2(\Omega)}^2.
\end{align*}
From Theorem \ref{theo:stateeq_diff} we know, that $y_h:= S'(u_{n'}^+)(\hat{u}-u_{n'}^+)$ is the weak solution of a uniquely solvable partial differential equation with right-hand side $\hat{u}-u_{n'}^+$. Hence, it is norm bounded by $c\norm{\hat{u}-u_{n'}^+}_{L^2(\Omega)}$ with $c>0$ independent of $n$.
With Young's Inequality we obtain
\begin{align*}
\norm{\mu_{n'}^+}_{L^1(\Omega)} +\frac{\alpha}{2\sigma}\norm{u_{n'}^+-\hat{u}}_{L^2(\Omega)}^2
&\leq  \frac{2}{\sigma}(\mu_{n'}^+,\psi-y_{n'}^+)_+  +\frac{4c}{\sigma\alpha}\norm{y_{n'}^+-y_d}_{L^2(\Omega)}^2 +\frac{\alpha}{\sigma}\norm{\hat{u}}_{L^2(\Omega)}^2.
\end{align*}
Exploiting the boundedness of $\norm{y_{n'}^+-y_d}_{L^2(\Omega)}$, $\hat{u}\in\Uad$, and Lemma \ref{lemma:scaprod_alg} this yields the assertion.
\end{proof}

Let us conclude this section with the following result on convergence.

\begin{theorem}[\textbf{Convergence towards KKT points}]
Assume that Algorithm \ref{alg_detail} generates the sequence $(y_n^+,u_n^+,p_n^+,\mu_n^+)_n$. Let $u^*$ denote a weak limit point of $(u_n^+)_n$.
If $u^*$ satisfies the linearized Slater condition from \eqref{eq:Slater_limitpoint}, then there exist subsequences $(y_{n'}^+,u_{n'}^+,p_{n'}^+,\mu_{n'}^+)_{n'}$ of $(y_n^+,u_n^+,p_n^+,\mu_n^+)_n$ such that
\begin{alignat*}{5}
u_{n'}^+ &\rightarrow u^* &\quad &\text{ in }L^2(\Omega),& \qquad &y_{n'}^+ \rightarrow y^* &\quad & \text{ in } H^1(\Omega)\cap \CO,\\
p_{n'}^+&\rightharpoonup p^* && \text{ in } W^{1,s}(\Omega),\ s\in[1,N/(N-1)) &&\mu_{n'}^+ \rightharpoonup^* \mu^* && \text{ in } \MO
\end{alignat*}
and $(y^*,u^*,p^*,\mu^*)$ is a KKT point of the original problem \eqref{eq:ALS:optcontprob}.
\label{theo:convKKT}
\end{theorem}
\begin{proof}
Since $(u_n^+)_n$ is bounded in $L^2(\Omega)$ we can extract a weak converging subsequence $u_{n'}^+\rightharpoonup u^*$ in $L^2(\Omega)$, thus $y_{n'}^+\rightarrow y^*$ in $H^1(\Omega)\cap\CO$ due to Theorem \ref{theo:stateeq}. Hence, \eqref{eq:kkt_o:1} ist satisfied. Since $u_{n'}^+$ satisfies a linearized Slater condition by Lemma \ref{lemma:SucStep-LinSlater} for $n'$ sufficiently large, Lemma \ref{lemma:multL1_boundedness} yields $L^1$-boundedness of $(\mu_{n'}^+)_{n'}$. Hence, we can extract a weak* converging subsequence in $\MO$ denoted w.l.o.g. by $\mu_{n'}\rightharpoonup^* \mu^*$, see \cite{HinzeMeyer2010lavrentiev}. Convergence of $p_{n'}^+ \rightharpoonup p^*$ in $W^{1,s}(\Omega),s\in[1,N/(N-1))$ can now be shown as in \cite[Lemma 11]{KrumbiegelNeitzelRoesch2012regSemilinOptCont}. Thus, the adjoint equation \eqref{eq:kkt_o:2} is satisfied. The space $W^{1,s}(\Omega)$,  is compactly embedded in $L^2(\Omega)$. Hence $p_{n'}^+ \rightarrow p^*$ in $L^2(\Omega)$ and we get
\begin{align*}
0&\leq \underset{n\rightarrow\infty}{\lim\inf}(p_{n'}^+ +\alpha u_{n'}^+,u-u_{n'}^+) = (p^*,u-u^*) - \underset{k\rightarrow\infty}{\lim\inf} (\alpha u_{n'}^+,u_{n'}^+-u)\\
& \leq (p^*,u-u^*) - (\alpha u^*,u^*-u)= (p^* + \alpha u^*,u-u^*),
\end{align*}
where we exploited the weak lower semicontinuity of $(\alpha u_{n'}^+,u-u_{n'}^+)$, $u\in L^2(\Omega)$. Hence, \eqref{eq:kkt_o:3} is satisfied.
Due to the structure of the update rule we have
$$\lim_{n'\rightarrow\infty}R_{n'}^+ = \lim_{n'\rightarrow\infty}\norm{(y_{n'}^+-\psi)_+}_\CO + (\mu_{n'}^+,\psi-y_{n'}^+)_+ = 0.$$
Hence $y^*\leq\psi$ and consequently $(\mu^*,\psi-y^*)\geq 0$. Since $(\mu^*,\psi-y^*)_+ =0$ we get $(\mu^*,\psi-y^*)= 0$. Thus \eqref{eq:kkt_o:4} is satisfied.
We have proven that $(y^*,u^*,p^*,\mu^*)$ is a KKT point of \eqref{eq:ALS:optcontprob}, i.e., $(y^*,u^*,p^*,\mu^*)$ solves \eqref{eq:kkt_o}. It remains to show strong convergence of $u_{n'}^+\rightarrow u^*$ in $L^2(\Omega)$.
Testing \eqref{eq:kkt_o:3} with $u_{n'}^+$ and \eqref{AL3_varineq} with $u^*$ and adding both inequalities we get
\begin{align*}
(p^*-p_{n'}^++\alpha(u^*-u_{n'}^+),u_{n'}^+-u^*)\geq 0.
\end{align*}
Hence $$\alpha \norm{u_{n'}^+-u^*}_{L^2(\Omega)}^2\leq (p^*-p_{n'}^+,u_{n'}^+-u^*).$$
Since we already know that $p_{n'}^+\rightarrow p^*$ in $L^2(\Omega)$ and $u_{n'}^+\rightharpoonup u^*$ in $L^2(\Omega)$ this directly yields $u_{n'}^+\rightarrow u^*$ in $L^2(\Omega)$.
\end{proof}

\section{Convergence towards Local Solutions}\label{sec:con_local_sol}

So far, we have been able to show that a weak limit point that has been generated by Algorithm \ref{alg_detail} is a stationary point of the original problem \eqref{eq:ALS:optcontprob} if it satisfies the linearized Slater condition.
If a weak limit point satisfies a second-order condition, we gain convergence to a local solution. However the convergence result from Theorem \ref{theo:convKKT} yields convergence of a subsequence of $(u_n^+)_n$ only. Accordingly, during all other steps the algorithm might choose solutions of the KKT system \eqref{AL:optsys} that are far away from a desired local minimum $\bu$. 
%Here two important questions are arising:
Here the following questions arise:
%Does a KKT point of the arising sub-problem exist for every fixed $\mu$ that satisfies $\bu_k\in B_r(\bu)$?\smallskip\\
\begin{itemize}
\item[1.] For every fixed $\mu$ does there exist a KKT point of the arising sub-problem that satisfies $\bu_k\in B_r(\bu)$?
\item[] and 
\item[2.] Is an infinite number of steps successful if the algorithm chooses these KKT points in step 1?
\end{itemize}
Indeed these questions can be answered positively. We will show in this section that for every fixed $\mu$  there exists a KKT point of the augmented Lagrange sub-problem such that for $\rho$ sufficiently large $\bu_k\in B_r(\bu)$. One should keep in mind, that also in this case there is no warranty that forces the algorithm to choose exactly these solutions. However, if the previous iterates are used in numerical computations as a starting point for the computation of the next iterate, the remaining iterates are likely located in $B_r(\bu)$. 
%Moreover, we will show that Assumption \ref{ass:termMuBounded} that has been the most crucial assumption in the previous section can be satisfied.\smallskip\newline
In order to reach this result we need the following assumption which is rather standard.
\begin{assumption}[\textbf{Quadratic growth condition (QGC)}]\label{ass:quadraticgrowthcondition}
Let $\bu\in \Uad$ be a control satisfying the first-order necessary optimality conditions \eqref{kkt_optsys}. We assume that there exist $\beta>0$ and $\epsilon>0$ such that the quadratic growth condition
\begin{align}\label{eq:quadgrowth}
f(u)\geq f(\bu)+\beta\norm{u-\bu}_{L^2(\Omega)}^2
\end{align}
is satisfied for all feasible $u\in \Uad$, $S(u)\leq\psi$ with $\norm{u-\bu}_{L^2(\Omega)} \leq\epsilon$. Hence, $\bu$ is a local solution in the sense of $L^2(\Omega)$ for problem \eqref{eq:ALS:optcontprob}.
\end{assumption}

Let us mention that the quadratic growth condition can be implied by some well known second-order sufficient condition (SSC). We refer the reader to Section \ref{sec:SSC_conv} for more  details.\medskip\newline
Our idea now is the following: In order to show that in every iteration of the algorithm there exists $\bu_k\in B_r(\bu)$ we want to estimate the error norm $\norm{\bu_k-\bu}_{L^2(\Omega)}^2$. Here we want to exploit the quadratic growth condition from Assumption \ref{ass:quadraticgrowthcondition}. However, this condition requires a control $u\in \Uad$ that is feasible for the original problem \eqref{eq:ALS:optcontprob}, which has explicit state constraints. Since the solutions of the augmented Lagrange sub-problems cannot be expected to be feasible for the original problem in general, we consider an auxiliary problem. Due to the special construction of this problem one can construct an auxiliary control that is feasible for the original problem \eqref{eq:ALS:optcontprob}. This idea has been presented in \cite{CasasTroeltzsch2002errorEstFinElementSemiLinOC} for a finite-element approximation as well as in \cite{KrumbiegelNeitzelRoesch2012regSemilinOptCont} for regularizing a semilinear elliptic optimal control problem with state constraints by applying a virtual control approach.

\subsection{Analysis of the Auxiliary Problem}

Let $\bu$ be a local solution of \eqref{eq:ALS:optcontprob} that satisfies the first-order necessary optimality conditions \eqref{kkt_optsys} of Theorem \ref{theo:KKT-AL} and the quadratic growth condition from Assumption \ref{ass:quadraticgrowthcondition}.
Following the idea from \cite{CasasTroeltzsch2002errorEstFinElementSemiLinOC,KrumbiegelNeitzelRoesch2012regSemilinOptCont} we consider the following auxiliary problem
\begin{align}
\underset{y^r_{\rho},u^r_{\rho}}{\min}\ &J_{AL}^r(y^r_{\rho},u^r_{\rho},\mu,\rho):=J(y^r_{\rho},u^r_\rho)+\frac{1}{2\rho}\int_{\Omega}\left(\left(\mu+\rho(y^r_{\rho}-\psi)\right)_+\right)^2 \dx
\label{auxprob_AL}\tag{${P_{AL}^r}$}
\end{align}
such that
\begin{align*}
 y^r_{\rho}=S(u^r_{\rho}),\qquad  u^r_{\rho}\in \Uad,\qquad \norm{u^r_{\rho}-\bu}_{L^2(\Omega)}\leq r.
\end{align*}
We choose $r$ small enough such that the quadratic growth condition from Assumption \ref{ass:quadraticgrowthcondition} is satisfied.
In the following we define the set of admissible controls of \eqref{auxprob_AL} by
\begin{align*}
\Uad^r:=\lbrace u\in \Uad\ |\ \norm{u-\bar{u}}_{L^2(\Omega)}\leq r\rbrace.
\end{align*}

The auxiliary problem admits at least one (global) solution. Moreover first-order necessary optimality conditions can be derived by standard arguments without any regularity assumption:

\begin{theorem}[\textbf{Existence of solution of the auxiliary problem}]
The auxiliary problem \eqref{auxprob_AL} admits a global solution $\bu^r_{\rho}\in\Uad^r$.
\end{theorem}
\begin{proof}
Can be found in \cite[Theorem 5.1]{Reyes2015numPDEopt}.
\end{proof}

\begin{theorem}[\textbf{Necessary optimality conditions of the auxiliary problem}]
Let $\bu^r_{\rho}$ be a local optimal solution of \eqref{auxprob_AL} and $\by^r_{\rho}$ its associated state. Then, there exist a unique adjoint state $\bp^r_{\rho}\in H^1(\Omega)$ and a unique Lagrange multiplier $\bmu^r_{\rho}\in \MO$ such that they satisfy the following optimality system
\begin{subequations}\label{AL:aux_optsys}
\begin{equation} \label{AL:aux1_adjoint}
\begin{alignedat}{2}
A \by^r_{\rho}+d(\by^r_{\rho})=\bu^r_{\rho}&\quad &\text{in } \Omega,\\
\partial_{\nu_A} \by^r_{\rho}=0 &&\text{on } \Gamma,
\end{alignedat}
\end{equation}
\begin{equation}\label{AL:aux2_pde}
\begin{alignedat}{2}
A^*\bp^r_{\rho} + d_y(\by^r_{\rho})\bp^r_{\rho}&= \by^r_{\rho}-y_d +\bmu^r_{\rho}  &\quad &\text{in } \Omega,\\
\partial_{\nu_{A^*}} \by^r_{\rho}&=0 &&\text{on } \Gamma,
\end{alignedat}
\end{equation}
\begin{equation} \label{AL:aux3_varineq}
(\bp^r_{\rho} +\alpha \bu^r_{\rho},u-\bu^r_{\rho})\geq 0,\qquad \forall u\in \Uad^r,
\end{equation}
\begin{equation} \label{AL:aux5_complcond}
\bmu^r_{\rho}=\left(\mu+\rho(y^r_{\rho}-\psi)\right)_+.
\end{equation}
\end{subequations}
\end{theorem}

\subsection{Construction of a Feasible Control}

In this section we want to construct a control $u^{r,\delta}\in \Uad^r$ that is feasible for the original problem \eqref{eq:ALS:optcontprob}, i.e., $u^{r,\delta}\in\Uad$ and $S(u^{r,\delta})\leq \psi$.
Based on a Slater point assumption controls of this type have already been constructed in \cite{Meyer2008errorEstFEapproxContProbStateCont} for obtaining error estimates of finite element approximation of linear elliptic state constrained optimal control problems. In \cite{KrumbiegelNeitzelRoesch2012regSemilinOptCont} these techniques were combined with the idea of the auxiliary problem presented for nonlinear optimal control problems in \cite{CasasTroeltzsch2002errorEstFinElementSemiLinOC}.\medskip\newline
We follow the strategy from \cite{KrumbiegelNeitzelRoesch2012regSemilinOptCont}. This work applied the virtual control approach in order to solve \eqref{eq:ALS:optcontprob}. This means, that the state constraints are relaxed in a suitable way. To obtain optimality conditions for the corresponding auxiliary problem the authors showed that the linearized Slater condition of the original problem can be carried over to feasible controls of the auxiliary problem. This transferred linearized Slater condition is also the main ingredient for the construction of feasible controls of the original problem.
In our case, the state constraints have been removed from the set of explicit constraints by augmentation. Thus it is not necessary to establish a linearized Slater condition for the auxiliary problem in order to establish optimality conditions. However the Slater-type inequality that is deduced in the following lemma is still needed for our analysis, see Lemma \ref{lemma:auxcontrolfeasible}.

\begin{lemma}\label{lemma:linslater}
Let $\bu$ satisfy Assumption \ref{ass:slater} with $\sigma>0$ and associated linearized Slater point $\hat{u}$. Let
$$\hat{u}^r := \bu + t(\hat{u}-\bu),\qquad t:=\frac{r}{\max(r,\norm{\hat{u}-\bu}_{L^2(\Omega)})},\qquad \sigma_r:=t\sigma.$$
Then, it holds $\norm{\hat{u}^r-\bu}_{L^2(\Omega)}\leq r$.
%
%$$\norm{\hat{u}^r-\bu}_{L^2(\Omega)}\leq r,\qquad \text{and } \qquad S(\bu) +S'(\bu)(\hat{u}^r-\bu)\leq \psi-\sigma_r.$$
%
Moreover, let $\bu^r_{\rho}\in \Uad^r$ be an admissible control of \eqref{auxprob_AL}. Then, for  $r>0$ sufficiently small  $\bu^r_{\rho}$ satisfies the following inequality
\begin{align}\label{eq:linSlaterAuxProb}
S(\bu^r_{\rho})+S'(\bu^r_{\rho})(\hat{u}^r-\bu^r_{\rho})\leq \psi - \frac{\sigma_r}{2}.
\end{align}
\end{lemma}
\begin{proof}
By definition of $\hat{u}^r$ and $t$ it holds $\norm{\hat{u}^r-\bu}_{L^2(\Omega)}\leq r$. Inserting the definition of $\hat{u}^r$ we get
\begin{align*}
S(\bu) +S'(\bu)(\hat{u}^r-\bu) & = S(\bu) + tS'(\bu)(\hat{u}-\bu)\\
&=(1-t)S(\bu) + t \left( S(\bu)+ S'(\bu)(\hat{u}-\bu)\right)\\
&\leq \psi -t\sigma =: \psi -\sigma_r.
\end{align*}
Hence, $\hat{u}^r$ is a linearized Slater point of the original problem \eqref{eq:ALS:optcontprob} in the neighborhood of $\bu$. We have $\norm{\hat{u}^r-\bu}\leq r, \norm{\bu-\bu^r_\rho}\leq r$ and hence $\norm{\hat{u}^r-\bu^r_{\rho}}\leq 2r$. Since $S$ and $S'$ are Lipschitz we obtain (if $r$ sufficiently small) $\norm{S(\bu^r_{\rho})-S(\bu)}_{\CO}\leq \sigma_r/6$,  $\norm{S'(\bu)(\bu-\bu^r_{\rho})}_\CO\leq \sigma_r/6$ and
$\norm{(S'(\bu^r_{\rho})-S'(\bu))(\hat{u}^r -\bu^r_{\rho}) }_\CO\leq \sigma_r/6$ . Hence,
\begin{align*}
S(\bu^r_{\rho})  + S'(\bu^r_{\rho})(\hat{u}^r-\bu^r_{\rho}) &=  S(\bu) + S'(\bu)(\hat{u}^r - \bu)\\
&\quad + S(\bu^r_{\rho})-S(\bu)\\
&\quad +(S'(\bu^r_{\rho})-S'(\bu))(\hat{u}^r -\bu^r_{\rho}) + S'(\bu)(\bu-\bu^r_{\rho})\\
&\leq \psi -\frac{\sigma_r}{2}.
\end{align*}
Thus, $\hat{u}^r$ satisfies \eqref{eq:linSlaterAuxProb} and the proof is done.\qedhere
\end{proof}

In the following lemma we will construct feasible controls for \eqref{eq:ALS:optcontprob} to be used in the sequel for our convergence analysis.
The construction of an admissible control $u^{r,\delta}\in\Uad^r$ that is also feasible for \eqref{eq:ALS:optcontprob} is based on the fact that $\bu^r_{\rho}$ satisfies Lemma \ref{lemma:linslater}.\medskip\newline
We define the maximal violation of $\bu^r_{\rho}$ with respect to the state constraints $\by^r_\rho\leq \psi$ by
\begin{align}\label{eq:def_max_constvio}
d[\bu^r_{\rho},(P)]:=\norm{(\by^r_{\rho}-\psi)_+}_\CO,
\end{align}
where $\by^r_{\rho}=S(\bu^r_{\rho})$.

\begin{lemma}\label{lemma:auxcontrolfeasible}
Let all assumptions from Lemma \ref{lemma:linslater} be satisfied and define $\delta_\rho\in(0,1)$ via
$$\delta_{\rho}:=\frac{d[\bu^r_{\rho},(P)]}{d[\bu^r_{\rho},(P)]+\frac{\sigma_r}{4}}.$$
Then, for every $\rho>0$ and $r>0$ small enough the auxiliary control
$$u^{r,\delta}:=\bu^r_{\rho}+\delta(\hat{u}^r-\bu^r_{\rho})$$
is feasible for the original problem \eqref{eq:ALS:optcontprob}, i.e., $S(u^{r,\delta})\leq \psi$ for all $\delta\in [0,\delta_{\rho}]$.
\end{lemma}
\begin{proof}
Applying \eqref{eq:linSlaterAuxProb} the proof follows the argumentation from \cite[Lemma 7]{KrumbiegelNeitzelRoesch2012regSemilinOptCont}.
\begin{comment}
Using the Taylor-Expansion
$$S(\bu^r_{\rho}+\delta(\hat{u}^r-\bu^r_{\rho}))=S(\bu^r_{\rho})+\delta \left(S'(\bu^r_{\rho})(\hat{u}^r-\bu^r_{\rho})\right)+\frac{\delta^2}{2}S''(u_{\delta})(\hat{u}^r-\bu^r_{\rho})^2,$$
with $u_{\delta}=\bu^r_{\rho}+\delta(\hat{u}^r-\bu^r_{\rho})$ for a $\delta\in(0,1)$ we get
\begin{align*}
S(u^{r,\delta})-\psi &= S(\bu^r_{\rho}+\delta(\hat{u}^r-\bu^r_{\rho}))-\psi\\
&=(1-\delta)(S(\bu^r_{\rho})-\psi)+ \delta(S(\bu^r_{\rho})-\psi+S'(\bu^r_{\rho})(\hat{u}^r-\bu^r_{\rho}))+\frac{\delta^2}{2}S''(u_{\delta})(\hat{u}^r-\bu^r_{\rho})^2\\
&\leq (1-\delta)d[\bu^r_{\rho},(P)]+ \delta(S(\bu^r_{\rho})+S'(\bu^r_{\rho})(\hat{u}^r-\bu^r_{\rho})-\psi)+\frac{\delta^2}{2}S''(u_{\delta})(\hat{u}^r-\bu^r_{\rho})^2\\
&\leq (1-\delta)d[\bu^r_{\rho},(P)] + \delta\left(-\frac{\sigma_r}{2} + Cr^2\right),
\end{align*}
where we used $\norm{\hat{u}^r-\bu^r_{\rho}}\leq 2r$ by the construction of $\hat{u}^r$ in Lemma \ref{lemma:linslater} and boundedness of $S''$ in $\CO$. Choosing $r>0$ small enough to satisfy $-\frac{\sigma_r}{2}+Cr^2\leq \frac{-\sigma_r}{4}$ and setting
$$ (1-\delta_{\rho} )d[\bu^r_{\rho},(P)] + \delta_{\rho} \frac{-\sigma_r}{4}=0$$
we get
$$\delta_{\rho} =\frac{d[\bu^r_{\rho},(P)]}{d[\bu^r_{\rho},(P)]+\frac{\sigma_r}{4}}$$
and hence,
$$S(u^{r,\delta})-\psi\leq 0,\qquad \forall \delta_\rho\leq \delta$$
Hence, our constructed auxiliary control $u^{r,\delta}$ is feasible for \eqref{eq:ALS:optcontprob}.
\end{comment}
\end{proof}

The error between the auxiliary control $u^{r,\delta}$ and the global solution $\bu^r_{\rho}$ of \eqref{auxprob_AL} is bounded by the maximal constraint violation.

\begin{lemma}\label{lemma:est:uk-udelta}
The constructed feasible control $u^{r,\delta}$ from Lemma \ref{lemma:auxcontrolfeasible} satisfies the estimate
$$\norm{\bu^r_{\rho}-u^{r,\delta}}_{L^2(\Omega)} \leq c d[\bu^r_{\rho},(P)].$$
\end{lemma}
\begin{proof}
We estimate $\delta_{\rho} $ from Lemma \ref{lemma:auxcontrolfeasible} by
\begin{align*}
\delta_{\rho} =\frac{d[\bu^r_{\rho},(P)]}{d[\bu^r_{\rho},(P)]+\frac{\sigma_r}{4}}\leq 4\frac{d[\bu^r_{\rho},(P)]}{\sigma_r}.
\end{align*}
Together with $\norm{\hat{u}^r-\bu^r_{\rho}}_{L^2(\Omega)}\leq 2r$ and the definition of $\sigma_r$ from Lemma \ref{lemma:linslater} as well as $\delta\in [0,\delta_{\rho}]$ we arrive at
\begin{align*}
\begin{split}
\norm{\bu^r_{\rho}-u^{r,\delta}}_{L^2(\Omega)} &=\norm{\bu^r_{\rho}-(\bu^r_{\rho}+\delta(\hat{u}^r-\bu^r_{\rho}))}_{L^2(\Omega)}=\norm{\delta(
\hat{u}^r-\bu^r_{\rho})}_{L^2(\Omega)}\\
&\leq \norm{\delta_{\rho}(\hat{u}^r-\bu^r_{\rho})}_{L^2(\Omega)} \leq 8r\frac{d[\bu^r_{\rho},(P)]}{\sigma_r} \\
&\leq 8 \frac{\max\lbrace r,\norm{\hat{u}^r-\bar{u}}_{L^2(\Omega)}\rbrace}{\sigma}d[\bu^r_{\rho},(P)] \leq c d[\bu^r_{\rho},(P)]
\end{split}
\end{align*}
and the proof is done.
\end{proof}

Finally we are able to apply the quadratic growth condition from Assumption \ref{ass:quadraticgrowthcondition}.

\begin{lemma}\label{lemma:erroryu}
Let $\bu$ be a local solution of \eqref{eq:ALS:optcontprob} that satisfies the quadratic growth condition from Assumption \ref{ass:quadraticgrowthcondition} and the linearized Slater condition from Assumption \ref{ass:slater}. Consider a fixed $\mu\in L^2(\Omega)$ and $r>0$ sufficiently small such that the quadratic growth condition is satisfied. If $\bu^r_{\rho}$ is a global solution of the auxiliary problem \eqref{auxprob_AL} then it holds
\begin{align}
\begin{split}
\beta\norm{\bu^r_{\rho}-\bu}^2_{L^2(\Omega)} +\frac{1}{2\rho}\norm{\bmu^r_{\rho}}_{L^2(\Omega)}^2 \leq c \norm{(\by^r_{\rho}-\psi)_+}_\CO +\frac{1}{2\rho}\norm{\mu}_{L^2(\Omega)}^2.
\end{split}
\label{eq:lemma:erroryu}
\end{align}
\end{lemma}
\begin{proof}
As has been shown in Lemma \ref{lemma:auxcontrolfeasible} $u^{r,\delta}$ is feasible for \eqref{eq:ALS:optcontprob}. We insert the special choice $u=u^{r,\delta}$ in the quadratic growth condition \eqref{eq:quadgrowth} and get
\begin{align}
f(u^{r,\delta})&\geq f(\bu)+\beta\norm{u^{r,\delta}-\bu}_{L^2(\Omega)}^2 \notag\\
&=f(\bu)+\beta\norm{u^{r,\delta}-\bu^r_{\rho}+\bu^r_{\rho}-\bu}_{L^2(\Omega)}^2 \notag\\
&\geq f(\bu)+\beta\left(\norm{u^{r,\delta}-\bu^r_{\rho}}_{L^2(\Omega)}^2-2|(u^{r,\delta}-\bu^r_{\rho},\bu^r_{\rho}-\bu)|+\norm{\bu^r_{\rho}-\bu}_{L^2(\Omega)}^2\right) \notag\\
&\geq f(\bu)+\beta \norm{\bu^r_{\rho}-\bu}_{L^2(\Omega)}^2-c\norm{\bu^r_{\rho}-u^{r,\delta}}_{L^2(\Omega)},
\label{eq:proof:lemmaerroryu1}
\end{align}
where we exploited that $ \norm{\bu^r_{\rho}-\bu}_{L^2(\Omega)}^2\leq r^2$ and $\norm{\bu^r_{\rho}-u^{r,\delta}}_{L^2(\Omega)} $ is bounded by the maximal constraint violation (Lemma \ref{lemma:est:uk-udelta}).
Rearranging the terms of \eqref{eq:proof:lemmaerroryu1} and applying Lemma \ref{lemma:est:uk-udelta} we get
\begin{align*}
\beta \norm{\bu^r_{\rho}-\bu}_{L^2(\Omega)}^2&\leq f(u^{r,\delta} )-f(\bu)+c\norm{u^{r,\delta} -\bu^r_{\rho}}_{L^2(\Omega)}\\
&\leq f(u^{r,\delta} )-f(\bu^r_{\rho})+f(\bu^r_{\rho})-f(\bu) +cd[\bu^r_{\rho},(P)].
\end{align*}

We recall the definition of the reduced cost functional of the auxiliary problem \eqref{auxprob_AL}
$$f_r(\bu^r_{\rho}) :=f(\bu_\rho^r)+\frac{1}{2\rho}\norm{\bmu_\rho^r}^2_{L^2(\Omega)}, \quad \bmu_\rho^r=(\mu+\rho(S(\bu_\rho^r)-\psi))_+.$$

Exploiting the Lipschitz continuity of the solution operator $S$ for the estimate
$$| f(u^{r,\delta} )-f(\bu^r_{\rho})|\leq c\norm{u^{r,\delta}-\bu^r_{\rho}}_{L^2(\Omega)},$$
see \cite[Lemma 4.11]{Troeltzsch2010optimal} and exploiting the optimality of $\bu^r_{\rho}$ for \eqref{auxprob_AL} as well as applying the definition of the reduced cost functional and the feasibility of $\bu$ for the auxiliary problem, we get
\begin{align*}
\beta &\norm{\bu^r_{\rho}-\bu}_{L^2(\Omega)}^2\leq f(\bu^r_{\rho})-f(\bu) +cd[\bu^r_{\rho},(P)]\\
&\quad\leq f_r(\bu^r_{\rho})-f_r(\bu) -\frac{1}{2\rho}\norm{\bmu^r_{\rho}}_{L^2(\Omega)}^2+\frac{1}{2\rho}\norm{(\mu+\rho(S(\bu)-\psi))_+}_{L^2(\Omega)}^2+cd[\bu^r_{\rho},(P)]\\
&\quad\leq -\frac{1}{2\rho}\norm{\bmu^r_{\rho}}_{L^2(\Omega)}^2 +\frac{1}{2\rho}\norm{(\mu+\rho(S(\bu)-\psi))_+}_{L^2(\Omega)}^2 + cd[\bu^r_{\rho},(P)].
\end{align*}
Noting that it holds
\begin{align*}
\frac{1}{2\rho}\norm{(\mu+\rho(S(\bu)-\psi))_+}_{L^2(\Omega)}^2\leq \frac{1}{2\rho}\norm{\mu}_{L^2(\Omega)}^2
\end{align*}
we get with \eqref{eq:def_max_constvio}
\begin{align*}
\beta \norm{\bu^r_{\rho}-\bu}_{L^2(\Omega)}^2+\frac{1}{2\rho}\norm{\bmu^r_{\rho}}_{L^2(\Omega)}^2&\leq cd[\bu^r_{\rho},(P)] +\frac{1}{2\rho}\norm{\mu}_{L^2(\Omega)}^2 \\
&= c\norm{(\by^r_{\rho}-\psi)_+}_\CO +\frac{1}{2\rho}\norm{\mu}_{L^2(\Omega)}^2
\end{align*}
which yields the claim.
\end{proof}

\subsection{An Estimate of the Maximal Constraint Violation}

In this section we will derive an estimate on the maximal constraint violation. We recall an estimate from \cite[Lemma 4]{KrumbiegelRoesch2008regerrorNeumann}.

\begin{lemma}\label{lemma:AL-fCest}
Let $f\in C^{0,1}(\bar{\Omega})$ be given. Then, there exists a constant $c>0$ so that $f$ satisfies the estimate
$$\norm{f}_\CO \leq c\norm{f}_{L^2(\Omega)}^\frac{2}{2+N}.$$
\end{lemma}

\begin{theorem}\label{theo:AL-maxconstraintviol}
Let $\mu\in L^2(\Omega)$ be fixed. Further, let $\bu^r_{\rho}$ be the optimal control of the auxiliary problem \eqref{auxprob_AL}. Then, the maximal violation $d[\bu^{r}_{\rho},(P)]$ of $\bu^r_{\rho}$ with respect to \eqref{eq:ALS:optcontprob} can be estimated by
\begin{align*}
d[\bu^r_{\rho},(P)]\leq c\left( \frac{1}{\rho}\right)^{1/(2+N)}.
\end{align*}
\end{theorem}
\begin{proof}
Since $\bu^r_{\rho}\in L^\infty(\Omega)$ we get with a regularity result \cite[Theorem 5]{KrumbiegelNeitzelRoesch2010moreauYosidaSemilinOptCont} that $\by^r_{\rho}\in W^{2,q}(\Omega)$ for all $1<q<\infty$. Due to the embedding $W^{2,q}(\Omega)\hookrightarrow C^{0,1}(\Omega)$ for $q>N$ we can apply Lemma \ref{lemma:AL-fCest} and get the following estimate
\begin{align*}
d[\bu^r_{\rho},(P)]&=\norm{(S(\bu^r_{\rho})-\psi)_+}_{\CO}\leq c\norm{(\by^r_{\rho}-\psi)_+}_{L^2(\Omega)}^{2/(2+N)}\\
&\leq  c\norm{\frac{1}{\rho}\left(\mu+\rho(\by^r_{\rho}-\psi)\right)_+}_{L^2(\Omega)}^{2/(2+N)}
 \leq c\left( \frac{1}{\rho}\norm{\bmu^r_{\rho}}_{L^2(\Omega)}\right)^{2/(2+N)}.
\end{align*}
Since $\mu\in L^2(\Omega)$ is fixed and $\norm{y^r_{\rho'}}_\CO\leq c\norm{u^r_{\rho'}}_{L^2(\Omega)}$ by Theorem \ref{theo:stateeq} we conclude with \eqref{eq:lemma:erroryu} from Lemma \ref{lemma:erroryu} that $\frac{1}{\rho}\norm{\bmu^r_{\rho}}^2$ is bounded.  Straight forward calculations yield
\begin{align*}
\left( \frac{1}{\rho}\norm{\bmu^r_{\rho}}_{L^2(\Omega)}\right)^{2/(2+N)}
&=\left( \frac{1}{\rho}\right)^{1/(2+N)}\left[ \frac{1}{\rho} \norm{\bmu^r_{\rho}}_{L^2(\Omega)}^{2}\right]^{1/(2+N)} \leq c\left( \frac{1}{\rho}\right)^{1/(2+N)}.
\end{align*}
Hence, we get the desired estimate.
\end{proof}

\subsection{Main Results}

We can now formulate our main results of this section.
\begin{theorem}
\label{theo:convergence}
Let $\bu$ be a local solution of \eqref{eq:ALS:optcontprob} with corresponding state $\by$ satisfying the QGC from Assumption \ref{ass:quadraticgrowthcondition} and the linearized Slater condition from Assumption \ref{ass:slater}. Let $\mu\in L^2(\Omega)$ be fix and let $(\by^r_{\rho},\bu^r_{\rho})$ denote the global solution of the auxiliary problem \eqref{auxprob_AL}.\smallskip\newline
Then, we have:
\begin{enumerate}
%\item [a)] The points $(\by^r_{\rho},\bu^r_{\rho},\bp^r_{\rho})$ are  KKT points of the augmented Lagrange sub-problem \probALk.
\item [a)] For every $r>0$ there is a $\bar\rho$ such that for all $\rho>\bar\rho$ it holds $\norm{\bu^r_{\rho}-\bu}_{\U}< r$.
\item [b)] The solutions $(\by^r_{\rho},\bu^r_{\rho})_\rho$ converge in $(\Y)\times L^2(\Omega)$ to $(\by,\bu)$ as $\rho\rightarrow\infty$ and we have the following convergence rates: Let $\gamma := \frac{1}{2(2+N)}$, then we have
\begin{align*}
\norm{\bu-\bu^r_{\rho}}_{L^2(\Omega)}&\leq \mathcal{O}\left(\frac{1}{{\rho}^\gamma}\right) \qquad \text{ and } \qquad \norm{\by-\by^r_{\rho}}_{\Y} \leq \mathcal{O}\left(\frac{1}{{\rho}^\gamma}\right).
\end{align*}
\item [c)] The points $\bu^r_{\rho}$ are local solutions of the augmented Lagrange sub-problem \probALk, provided that $\rho$ is sufficiently large.
\end{enumerate}

\end{theorem}

\begin{proof} %a) Since $\Uad^r\subset \Uad$, we immediately get the first statement.\smallskip\\
a) + b) From Lemma \ref{lemma:erroryu}, the estimate of the maximal constraint violation from Theorem \ref{theo:AL-maxconstraintviol} and the Lipschitz continuity of the solution operator \eqref{eq:S-Lipschitz} we get the following error estimate
\begin{align*}
\norm{\by^r_{\rho}-\by}_{\Y}^2 +\norm{\bu^r_{\rho}-\bu}_{L^2(\Omega)}^2+\frac{1}{2\rho}\norm{\bmu^r_{\rho}}_{L^2(\Omega)}^2\leq c\left( \frac{1}{\rho}\right)^{1/(2+N)}+\frac{1}{2\rho}\norm{\mu}^2_{L^2(\Omega)}
\end{align*}
and we can conclude the existence of  $\bar\rho, r>0$ such that for all $\rho>\bar\rho$ we have $\norm{\bu^r_{\rho}-\bu}_{L^2(\Omega)}<r$. \smallskip\\
c) We have to show that
\begin{align*}
f_{AL}(u)\geq f_{AL}(\bu_\rho^r)\quad \forall u\in \Uad \text{ with } \norm{u-\bu_\rho^r}_{L^2(\Omega)}\leq \frac{r}{2}
\end{align*}
holds for a certain $r>0$. Since $\bu_\rho^r$ is the global solution of the auxiliary problem \eqref{auxprob_AL} we already know that there holds
\begin{align*}
f_{AL}(u)\geq f_{AL}(\bu_\rho^r)\quad \forall u\in \Uad \text{ with } \norm{u-\bu}_{L^2(\Omega)}\leq r.
\end{align*}
Let now $u\in \Uad$ such that $\norm{u-\bu_\rho^r}_{L^2(\Omega)}\leq \frac{r}{2}$. The triangle inequality yields
\begin{align*}
\norm{u-\bu}_{L^2(\Omega)}\leq \norm{u-\bu_\rho^r}_{L^2(\Omega)}+\norm{\bu_\rho^r-\bu}_{L^2(\Omega)}\leq \frac{r}{2}+\frac{r}{2}=r
\end{align*}
for $\rho$ sufficiently large. Here, we exploited statement b). Hence, $u\in \Uad^r$ where $f_{AL}(u)\geq f_{AL}(\bu_\rho^r)$ is satisfied. By definition we can conclude that $\bu_\rho^r$ is a local solution of \probALk.
\end{proof}

We can further prove that the algorithm makes infinitely many successful steps if $(\by_k,\bu_k)$ in step 1 of Algorithm \ref{alg_detail} are chosen as the global minimizers of the corresponding auxiliary problem.

\begin{theorem}\label{theo:Ass3-auxprob}
Assume that in step 1 of Algorithm \ref{alg_detail}
$(\by_k,\bu_k,\bp_k)$ is chosen as the global solution of the auxiliary problem \eqref{auxprob_AL} if it solves the optimality system of the augmented Lagrange sub-problem \eqref{AL:optsys}. Assume that only finitely many steps of Algorithm \ref{alg_detail} are successful. Then Assumption \ref{ass:termMuBounded} is satisfied.
\end{theorem}
\begin{proof}
Let $m$ denote the largest index of a successful step. Hence $\mu_k=\mu_m$ for all $k>m$. The sequence $(\rho_k)_k$ is monotonically increasing. Exploiting Theorem \ref{theo:convergence} c) we can find an index $K>m$ such that for all $k>K$ the global solution $(\by_k,\bu_k)$ of the auxiliary problem is a KKT point of \eqref{AL:optsys}. Further due to Lemma \ref{lemma:erroryu} and Theorem \ref{theo:AL-maxconstraintviol} the following inequality is satisfied
\begin{align*}
\frac{1}{2\rho_{k}}\norm{\bmu_{k}}_{L^2(\Omega)}^2 &\leq c\norm{(\by_{k}-\psi)_+}_\CO +\frac{1}{2\rho_{k}}\norm{\mu_k}_{L^2(\Omega)}^2\\
 &\leq c\left(\frac{1}{\rho_{k}}\right)^{1/(2+N)}+\frac{1}{2\rho_{k}}\norm{\mu_m}_{L^2(\Omega)}^2\\
  &\leq c\left(\frac{1}{\rho_{1}}\right)^{1/(2+N)}+\frac{1}{2\rho_{1}}\norm{\mu_m}_{L^2(\Omega)}^2.
\end{align*}
Hence, Assumption \ref{ass:termMuBounded} is satisfied.
\end{proof}

We can conclude that the algorithm makes infinitely many successful steps. We omit the proof since it uses the same arguments as in Lemma \ref{lemma:globsol_infsuc}.

\begin{corollary}
Let all assumptions from Theorem \ref{theo:Ass3-auxprob} be satisfied. Then Algorithm \ref{alg_detail} makes infinitely many successful steps.
\end{corollary}

One has to keep in mind that the quadratic growth condition is only a local condition. Hence, the result of Theorem \ref{theo:convergence} is actually the best we can expect. In particular, the sub-problems \probALk may have solutions arbitrarily far from $\bu$ and we cannot exclude the possibility that these solutions are chosen in the sub-problem solution process from Algorithm \ref{alg_detail}.
However, one can prevent this kind of scenario by using the previous iterate $\bu_k$ as a starting point for the computation of $\bu_{k+1}$. In this way it is reasonable to expect that as soon as one of the iterates $\bu_k$ lies in $B_r(\bu)$ (with $r$ as above) and the penalty
parameter is sufficiently large, the remaining iterates will stay in $B_r(\bu)$ and converge to $\bu$.

\section{Second-Order Sufficient Conditions}\label{sec:SSC_conv}

We take up the quadratic growth condition from Assumption \ref{ass:quadraticgrowthcondition}. This condition is implied by a second-order sufficient condition, see \cite{CasasReyesTroeltzsch2008soscSemilinState}. We define the Lagrangian function
$$\underset{u\in \Uad}{\min} \mathcal{L}(u,\mu)=f(u)+\int_{\bar{\Omega}} \left(S(u)-\psi\right) \,\mathrm{d}\mu$$
%Using the definition of the adjoint state $p$, the second derivative can be expressed via
%\begin{align}
%\begin{split}
%\frac{\partial^2\mathcal{L}}{\partial u^2} (u,\mu)[h_1,h_2]&= f''(u)[h_1,h_2] + \int_{\bar{\Omega}}S''(u)[h_1,h_2]\,\mathrm{d}\mu\\
%&= \int_\Omega (y_{h_1}y_{h_2}+\alpha h_1h_2-pd_{yy}(x,y(u))y_{h_1}y_{h_2})\dx,
%\end{split}
%\end{align}
where $y=S(u)$ 
%and $y_{h_i}=S'(u)h_i$, $i=1,2$ and $p$ is solution of the adjoint equation \eqref{eq:kkt_o:2}.
and assume that for all $(\by,\bp,\bmu)$ satisfying the first-order necessary optimality conditions \eqref{kkt_optsys} to $\bu$
it holds
\begin{align}
%\frac{\partial^2 \mathcal L}{\partial u^2}(\bu,\bmu)[h,h]\geq \delta\norm{h}^2_{L^2(\Omega)},\qquad \forall h\in {L^2(\Omega)}.
\frac{\partial^2 \mathcal L}{\partial u^2}(\bu,\bmu)[h,h]\geq 0,\qquad \forall h \in C_{\bu}\backslash\lbrace0\rbrace,
\label{eq:SSC}
\end{align}
where $C_{\bu}$ denotes the cone of critical directions as defined in \cite{CasasReyesTroeltzsch2008soscSemilinState}. Since the solution operator $S$ (Theorem \ref{theo:stateeq_diff}) and the cost functional $J:L^2(\Omega)\rightarrow \R$ are of class $C^2$ (see\cite{CasasMateos2002soscSemiLinFiniteConstraints,CasasReyesTroeltzsch2008soscSemilinState}), inequality \eqref{eq:SSC} together with the first-order necessary conditions implies the quadratic growth condition from Assumption \ref{ass:quadraticgrowthcondition}, see \cite[Theorem 4.1, Remark 4.2]{CasasReyesTroeltzsch2008soscSemilinState} and \cite{Troeltzsch2010optimal}. Note, that the multiplier $\bmu$ does not need to be unique. That is why \eqref{eq:SSC} is imposed for every multiplier.\medskip\newline
Let us return to the convergence analysis of Algorithm \ref{alg_detail}. If in addition to the assumptions of Theorem \ref{theo:convKKT}, $u^*$ satisfies the QGC from Assumption \ref{ass:quadraticgrowthcondition}, then $u^*$ obviously is a local solution.\smallskip\newline
Second-order sufficient conditions not only allow us to prove convergence to a local solution but also to show local uniqueness of stationary points of the augmented Lagrange sub-problem. This is an important issue for numerical methods. In \cite{KrumbiegelNeitzelRoesch2010moreauYosidaSemilinOptCont} the authors proved that the Moreau-Yosida regularization without additional shift parameter is equivalent to the virtual control problem for a specific choice of therein appearing parameters. This equivalence can be transferred to the augmented Lagrange sub-problem \eqref{prob_auglag}.

\begin{remark}
Let $\bu\in \Uad$ be a control that satisfies the first-order necessary optimality conditions \eqref{kkt_optsys} and let $\bmu$ be the unique Lagrange multiplier w.r.t. the state constraints. We assume that there exists a constant $\delta>0$ such that
\begin{align}
\frac{\partial^2 \mathcal L}{\partial u^2}(\bu,\bmu)[h,h]\geq \delta\norm{h}^2_{L^2(\Omega)},\qquad \forall h\in {L^2(\Omega)}.
\label{eq:SSCstrong}
\end{align}

One can prove that the SSC \eqref{eq:SSCstrong} can be carried over to the augmented Lagrange sub-problems.
Let $\mu\in L^2( \Omega)$ and $\rho>0$ be fixed. Let $\bu_\rho\in \Uad$  be a control that satisfies $\bu_\rho\in B_r(\bu)$ and the first-order necessary optimality conditions \eqref{AL:optsys}. Let the SSC \eqref{eq:SSCstrong} be satisfied. Then, there exists a constant $\delta'>0$, which is independent of $\mu$ such that for all $h\in L^2(\Omega)$ the following condition
\begin{align*}
f''(\bu_\rho)h^2+((\mu+\rho(S(\bu_\rho)-\psi)_+,S''(\bu_\rho)h^2)\geq \delta'\norm{h}_{L^2(\Omega)}^2
\end{align*}
or equivalently
\begin{align*}
\int_{\Omega}(y^2_{h}-\bp_\rho d_{yy}(x,\by_\rho)y^2_{h}+\alpha h^2)\ \dx\geq \delta'\norm{h}_{L^2(\Omega)}^2
\end{align*}
is fulfilled for all $(h,y_{h})\in L^2(\Omega)\times H^1(\Omega)$ provided that $\rho$ is sufficiently large. Here, $y_{h}=S'(\bu_\rho)h$ and $\bp_\rho$ is the solution of the adjoint equation of the augmented Lagrange sub-problem. \newline\medskip
Moreover, then there exists a constant $\beta>0$ and $\gamma>0$ such that the quadratic growth condition
\begin{align*}
f_{AL}(u)\geq f_{AL}(\bu_\rho)+\beta\norm{u-\bu_\rho}^2_{L^2(\Omega)}
\end{align*}
holds for all $u\in\Uad$ with $\norm{u-\bu_\rho}_{L^2(\Omega)}\leq \gamma$ and $\bu_\rho$ is a local solution with corresponding state $\by_\rho$ of the augmented Lagrange sub-problem.
Here, Theorem 13 from \cite{KrumbiegelNeitzelRoesch2012regSemilinOptCont} yields the carried over version of the second-order condition for a virtual control problem. In \cite[Proposition 3]{KrumbiegelNeitzelRoesch2010moreauYosidaSemilinOptCont} it is proved that this condition implies a quadratic growth condition for the virtual control problem. Further, following the arguments as in \cite[Theorem 5]{KrumbiegelNeitzelRoesch2010moreauYosidaSemilinOptCont} this results can be adapted to the augmented Lagrange sub-problem.
\end{remark}

\section{Numerical Tests}\label{sec:NumExp}
In this section we report on numerical results for the solution of a semilinear elliptic pointwise state
constrained optimal control problem in two dimensions. All optimal control problems have been
solved using the above stated augmented Lagrange algorithm implemented with FEniCS \cite{LoggMardalEtAl2012a}
using the DOLFIN \cite{LoggWells2010a} Python interface. \smallskip\\
In every outer iteration of the augmented Lagrange algorithm the KKT system \eqref{AL:optsys} has to be solved for given $\mu$ and $\rho$. This is done by applying a semi-smooth Newton method.
We define the sets
\begin{align}
\begin{split}
\Aa_{\rho}:=\bigg\lbrace x\in \Omega \colon  -&\frac{1}{\alpha}\bp_\rho\leq u_a \bigg\rbrace,\qquad \Ab_{\rho}:= \left\lbrace x\in \Omega \colon -\frac{1}{\alpha}\bp_\rho\geq u_b\right\rbrace,\\
&\Yset_{\rho}:=\left\lbrace x\in \Omega \colon (\mu+\rho(\by_\rho-\psi))(x) > 0 \right\rbrace.
\end{split}
\label{eq:activesets}
\end{align}

Then system \eqref{AL:optsys} can be stated as
\begin{align}\label{eq:KKT_SSN_activeSets}
\begin{split}
A\by_{\rho}+d(\by_{\rho}) &= \bu_{\rho} \\
A^*\bar{p}_{\rho} + d_y(\by_{\rho})\bp_{\rho} & =  \bar{y}_{\rho}-y_d +\chi_{\Yset_{\rho}}\left(\mu+\rho(\by_{\rho}-\psi)\right)\\
\bu_\rho + (1-\chi_{\Aa_{\rho}}-\chi_{\Ab_{\rho}}) \frac{1}{\alpha}\bp_\rho &= \chi_{\Aa_{\rho}} u_a + \chi_{\Ab_{\rho}} u_b.
\end{split}
\end{align}

The semi-smooth Newton method for solving \eqref{AL:optsys} is given in Algorithm  \ref{alg:semi-smooth}.

\begin{algorithm}
    \caption{Semi-smooth Newton method for the augmented Lagrange sub-problem}
    \label{alg:semi-smooth}
    \begin{algorithmic}
     \item[1:] Set $k=0$, $\rho>0, \alpha>0$, set $\mu\in L^2(\Omega), y_d\in L^2(\Omega), \psi\in \CO$.\\
     Choose $(y_0,u_0,p_0)$ in $\Y\times\U\times H^1(\Omega)$\\
     \item[2:] \textbf{repeat}
        \item[3:] Set $\Aa_k, \Ab_k$and $\Yset_k$ as defined in \eqref{eq:activesets}
        \item[4:] Solve for $\delta_y,\delta_u,\delta_p$ by solving
\begin{align*}
G(y_k,u_k,p_k)(\delta_y,\delta_u,\delta_p) = - F(y_k,u_k,p_k)
\end{align*}
where
\begin{align*}
G(y_k,u_k,p_k) :=
\begin{pmatrix}
A+d_y(y_k) & -\mathrm{Id} & 0 \\
-(\mathrm{Id}+\chi_{\Yset_{k}}\rho\cdot \mathrm{Id}) +d_{yy}(y_k)p_k & 0 & A^*+d_y(y_k) \\
0 & \mathrm{Id} &  \frac{1}{\alpha}(1-\chi_{\Aa_{k}}-\chi_{\Ab_{k}})
\end{pmatrix}
\end{align*}
and
\begin{align*}
F(y_k,u_k,p_k) :=
\begin{pmatrix}
Ay_k+d(y_k) - u_k \\
A^*p_k+ d_y(y_k)p_k -  y_k + y_d - \chi_{\Yset_{k}}\left(\mu+\rho(y_k-\psi)\right) \\
u_k + (1-\chi_{\Aa_{k}}-\chi_{\Ab_{k}})\frac{1}{\alpha}p_k - \chi_{\Aa_{k}} u_a - \chi_{\Ab_{k}} u_b
\end{pmatrix}
\end{align*}
        \item[5:] Set $y_{k+1}=: y_k +\delta_y, u_{k+1}:= u_k+\delta_u$ and $p_{k+1}:= p_k+\delta_p,$
        \item[6:] Set $k:=k+1$.
        \item[7:] \textbf{until} a suitable stopping criterion is satisfied.
    \end{algorithmic}
\end{algorithm}

Since the linear parts of the system can be solved exactly we choose the error that arises during the linearization of the discretized system \eqref{eq:KKT_SSN_activeSets} as a stopping criterion. We terminate the semi-smooth Newton method as soon as 
$$\max( r_1, r_2 ,r_3) \leq 10^{-6},$$
where
\begin{small} 
\begin{align*}
r_1 &:= \norm{ d(y_{k})-\left(   d_y(y_{k-1})(y_{k}-y_{k-1})+d(y_{k-1})  \right)},\\
r_2 &:=\norm{d_y(y_{k}) - (d_y(y_{k-1})p_{k} + d_{yy}(y_{k-1})p_{k-1}(y_{k}-y_{k-1})) + (\chi_{\Yset_{k}}-\chi_{\Yset_{k-1}})(\mu+\rho(y_{k}-\psi)},\\
r_3 &:=\norm{u_k-P_{\Uad}\left(-\frac{1}{\alpha}p_k\right) }
\end{align*} 
\end{small}

is satisfied. In the following, $(y_h , u_h , p_h, \mu_h )$ denote the calculated
solutions after the stopping criterion is reached. We consider optimal control problems like
\begin{align*}
\min\ J(y,u):&=\frac{1}{2}||y-y_d||_{L^2(\Omega)}^2+\frac{\alpha}{2}||u||_{L^2(\Omega)}^2\\
\text{s.t.}\quad & y=Su,\qquad y\leq \psi,\qquad u\in \Uad
\end{align*}
where $\Omega = [0,1]\times [0,1]$. As not mentioned otherwise, we initialize  $(\by_0,\bu_0,\bp_0,\mu_1)$ equal to zero, the penalty parameter with $\rho_0:=0.5$ and choose the parameter in the decision concerning successful steps to be $\tau :=0.1$. If a step has not been successful, the penalization parameter is increased by the factor $\theta:=10$.
We stopped the algorithm as soon as
\[
R_n^+:=\norm{(y_n^+-\psi)_+}_{C(\bar{\Omega})} + ( \mu_n^+,\psi-y_n^+)_+ \leq  10^{-6}
\]
was satisfied. Since the stopping criterion from Algorithm \ref{alg:semi-smooth} yields $(y_h,u_h,p_h)$ that satisfies \eqref{eq:kkt_o:1}-\eqref{eq:kkt_o:3} with the desired accuracy this is a suitable stopping criterion.

\subsection*{Example 1}
Let us first consider an optimal control problem that is governed by the following partial differential equation
\begin{alignat*}{2}
-\Delta y +y+ \exp(y) &= u &\quad &\text{ in } \Omega,\\
\partial_{\nu} y &= 0 &&\text{ on } \Gamma.
\end{alignat*}
Clearly $d(y):=\exp(y)$ satisfies the required assumptions from Assumption \ref{ass:standing}. We set $$y_d(x) := 8\sin(\pi x_1)\sin(\pi x_2)-4,$$ $\psi(x):=1.0$ and $\Uad := \left\lbrace u\in L^\infty(\Omega) \colon -100\leq u(x)\leq 200\right\rbrace$. We choose $\alpha := 10^{-5}$. Figures \ref{fig:Ex5-100-y} and \ref{fig:Ex5-100-mu} illustrate the computed results for a degree of freedom of $10^4$.

\begin{figure}[H]
\begin{minipage}[h]{0.48\textwidth}
\includegraphics[width=\textwidth]{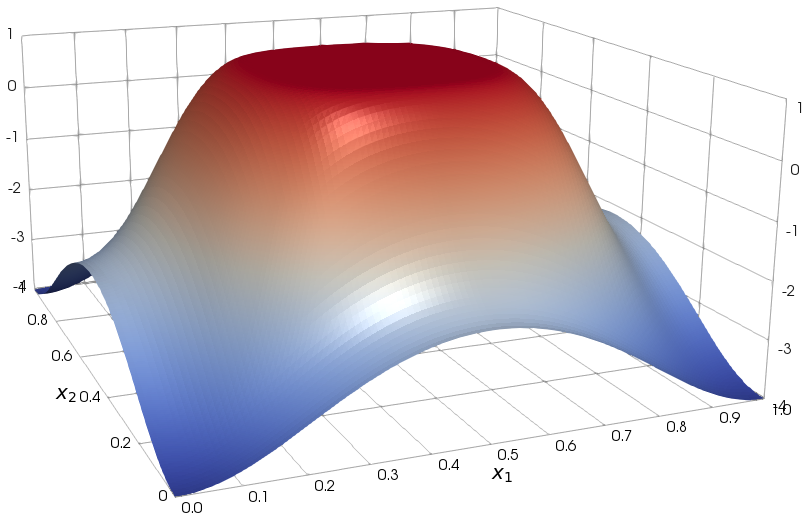}
\end{minipage}
\hfill
\begin{minipage}[h]{0.48\textwidth}
\includegraphics[width=\textwidth]{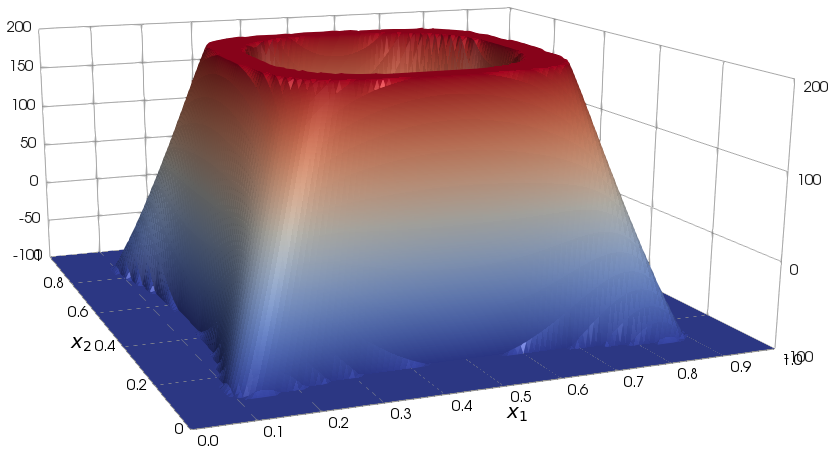}
\end{minipage}
\caption{(Example 1) Computed discrete optimal state $y_h$ (left) and optimal control $u_h$ (right)}
\label{fig:Ex5-100-y}
\end{figure}
\begin{figure}[H]
\begin{minipage}[b]{0.48\textwidth}
\includegraphics[width=\textwidth]{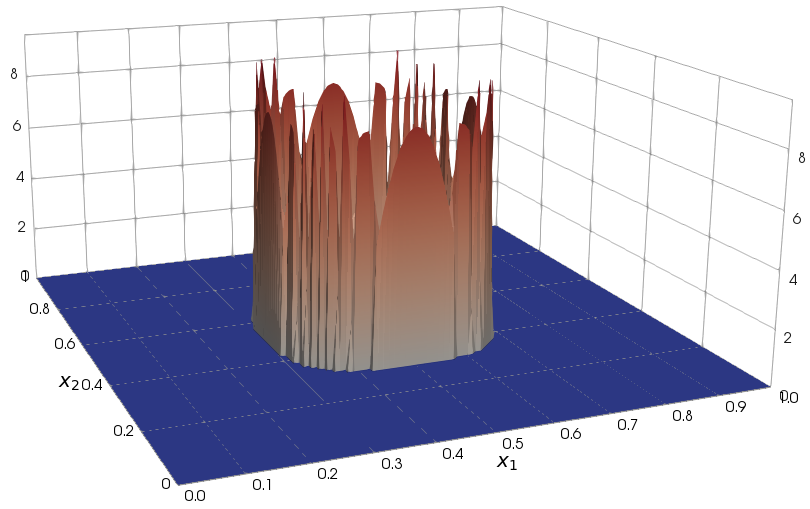}
\end{minipage}
\hfill
\begin{minipage}[b]{0.48\textwidth}
\includegraphics[width=\textwidth]{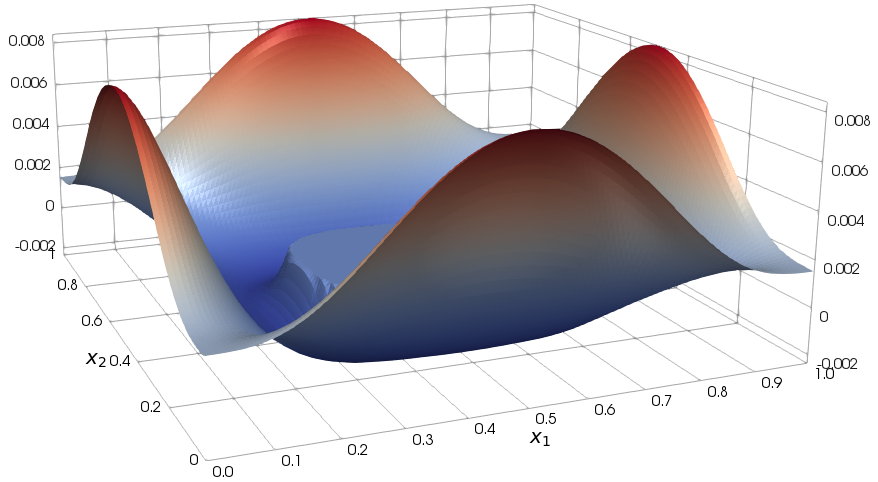}
\end{minipage}
\caption{(Example 1) Computed discrete multiplier $\mu_h$ (left) and the adjoint state $p_h$ (right)}
\label{fig:Ex5-100-mu}
\end{figure}

\subsection*{Example 2}
Next, we consider the following partial differential equation
\begin{alignat*}{2}
-\Delta y + y^3 &= u+f &\quad &\text{ in } \Omega,\\
\partial_{\nu} y &= 0 &&\text{ on } \Gamma
\end{alignat*}
and construct $(\by,\bu,\bp,\bmu)$ that satisfy the KKT system \eqref{eq:kkt_o}.
Let $\Omega:=B_2(0)$. We consider box constraints and set $u_a := -5$, $u_b := 5$. For clarity and to shorten our notation we set $r := r(x_1,x_2) := \sqrt{x_1^2 + x_2^2}$ and define the following functions
\begin{align*}
\bar y(x_1,x_2) &:= \begin{cases} 1 & \text{if } r < 1\\ 32 - 120 \cdot r + 180 \cdot r^2 - 130 \cdot r^3 + 45 \cdot r^4 - 6 \cdot r^5 & \text{if } r \geq 1 \end{cases},\\
\bar p(x_1,x_2) &:= 2\cos\left(\frac{3}{4}\pi x_1\right) \cos\left(\frac{3}{4}\pi x_2\right) \cdot \left( 1 - \frac{5}{4}r^3 + \frac{15}{16} r^4 - \frac{3}{16} r^5 \right),\\
\bar u(x_1,x_2) &:= P_{\Uad}\left(-\frac{1}{\alpha}\bp(x_1,x_2)\right),\\
\bar \mu(x_1,x_2) &:=   \begin{cases} \exp\left( - \frac{1}{1-r^2} \right)  & \text{if } r < 1\\ 0 & \text{if } r \geq 1 \end{cases},\\
\psi(x_1,x_2) &:= 1.
\end{align*}
Some calculation show that $\by, \bar p \in C^2(\bar \Omega)$ and $\bar \mu \in C(\bar \Omega)$. Furthermore $\partial_\nu \by = \partial_\nu \bp = 0$ on $\Gamma$. We now set
\begin{align*}
f(x_1,x_2) &:= - \Delta \bar y(x_1,x_2) +\by^3(x_1,x_2)- \bar u(x_1,x_2),\\
y_d(x_1,x_2) &:= \Delta \bar p(x_1,x_2) -3\by^2(x_1,x_2)\bp(x_1,x_2)+ \bar y(x_1,x_2) + \bar \mu(x_1,x_2).
\end{align*}

We start the algorithm with $\rho_0:=1$ and $\tau:= 0.5$. The Figures \ref{fig:Ex6-y} and \ref{fig:Ex6-mu} depict the computed result for a degree of freedom of $10^4$. Moreover, Figure \ref{fig:Ex2-erroryu} depicts the $L^2$-error of the computed solution $(y_h,u_h,p_h)$ to the constructed solution $(\by,\bu,\bp)$ in dependence of the degrees of freedom.

\begin{figure}[H]
\begin{minipage}[h]{0.48\textwidth}
\includegraphics[width=\textwidth]{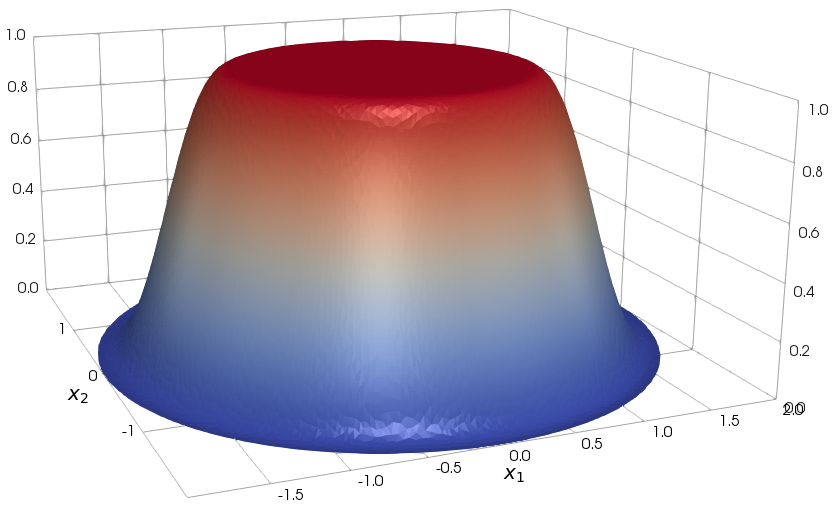}
\end{minipage}
\hfill
\begin{minipage}[h]{0.48\textwidth}
\includegraphics[width=\textwidth]{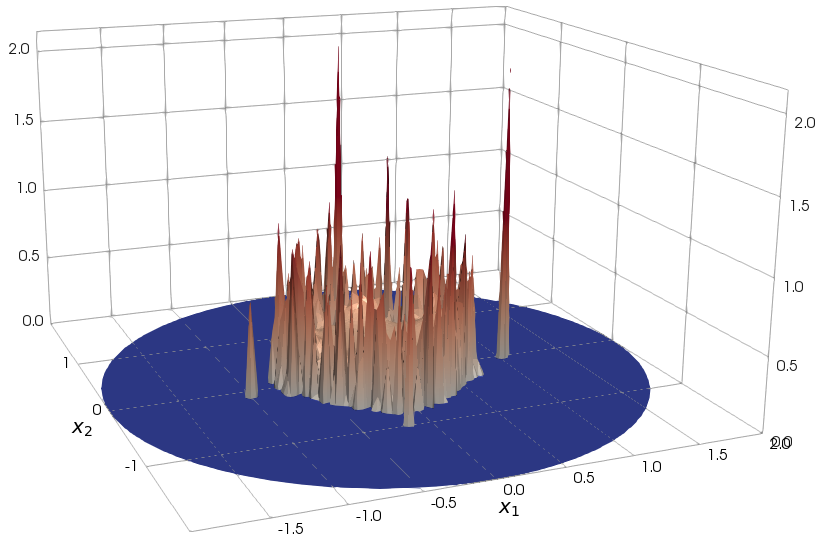}
\end{minipage}
\caption{(Example 2) Computed discrete optimal state $y_h$ (left) and multiplier $\mu_h$ (right)}
\label{fig:Ex6-y}
\end{figure}
\begin{figure}[H]
\begin{minipage}[b]{0.48\textwidth}
\includegraphics[width=\textwidth]{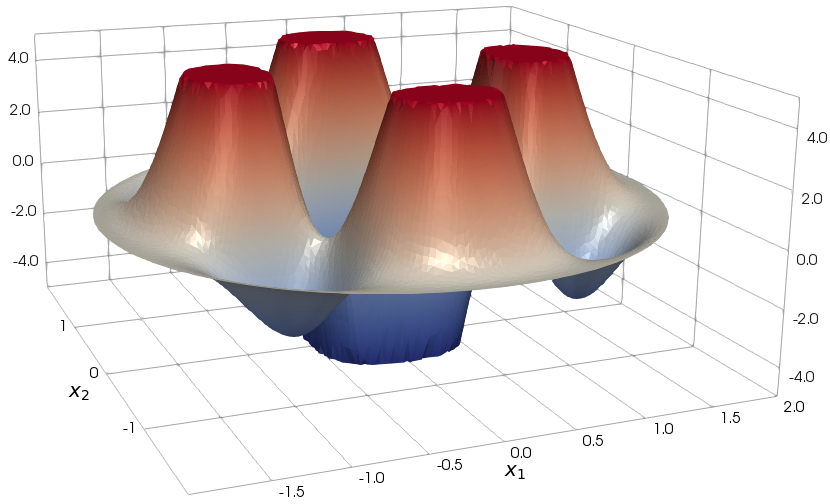}
\end{minipage}
\hfill
\begin{minipage}[b]{0.48\textwidth}
\includegraphics[width=\textwidth]{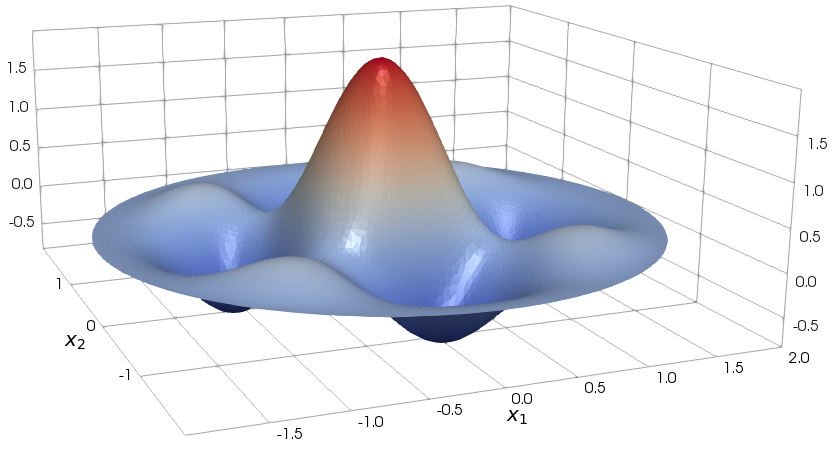}
\end{minipage}
\caption{(Example 2) Computed discrete optimal control $u_h$ (left) and the adjoint state $p_h$ (right)}
\label{fig:Ex6-mu}
\end{figure}

\begin{figure}[ht]
\begin{minipage}[b]{\textwidth}
\begin{figure}[H]
\centering
\newlength\figureheight
\newlength\figurewidth
\setlength\figureheight{5cm}
\setlength\figurewidth{8cm}
\input{Ex2-erroryu.tikz}
\end{figure}
\end{minipage}
\caption{(Example 2): Errors $\norm{u_h-\bu}_{L^2(\Omega)},\norm{y_h-\by}_{L^2(\Omega)}$ and $\norm{p_h-\bp}_{L^2(\Omega)}$ vs. degrees of freedom.}
\label{fig:Ex2-erroryu}
\end{figure}

\subsection*{Example 3}

We adapt an example from \cite{KarlWachsmuth2017augmented} which can also be found in \cite{RoeschWachsmuth2012aposteriori} for state constraints given by $y\geq \psi$. In this case $\Omega:= [-1,2]\times [-1,2]$. This example does not include constraints on the control. The optimal control problem is governed by the semilinear partial differential equation
\begin{alignat*}{2}
-\Delta y + y^5 &= u+f &\quad &\text{ in } \Omega,\\
\partial_{\nu} y &= 0 &&\text{ on } \Gamma
\end{alignat*}
which satisfies Assumption \ref{ass:standing}. We set $r:= r(x_1,x_2):= \sqrt{x_1^2+x_2^2}$. The state constraint is given by $\psi(r) := -\frac{1}{2\pi\alpha}\left(\frac{1}{4}-\frac{r}{2}\right)$. Further, we have
\begin{align*}
\by(r) := -\frac{1}{2\pi\alpha}\chi_{r\leq 1}\left(\frac{r^2}{4}(\log r-2)+\frac{r^3}{4}+\frac{1}{4}\right),\qquad \bu(r) := \frac{1}{2\pi\alpha}\chi_{r\leq 1}(\log r +r^2-r^3),
\end{align*}
\[
\bp(r) := -\alpha \bu(r),\qquad \bmu(r) := \delta_0(r).
\]
It can be checked easily that $\by$ and $\bp$ satisfy the Neumann boundary. We consider the auxiliary functions
\begin{align*}
\tilde{y}_d(r) &:= \by(r)-\frac{1}{2\pi}\chi_{r\leq 1}(4-9r),\quad  \tilde{f}(r) := -\frac{1}{8\pi}\chi_{r\leq 1}(4-9r+4r^2-4r^3)
\end{align*}
and set
\begin{align*}
y_d(r) := \tilde{y}_d(r) - 5\by^4\bp, \quad f(r):= \tilde{f}(r)-\by^5.
\end{align*}
We start the algorithm with $\alpha:= 1.0$, $\rho_0:= 0.5$ and $\tau:= 0.3$. The computed results can be seen in Figures \ref{fig:Ex9-100-y} and \ref{fig:Ex9-100-mu}. The $L^2$-error of the computed solution $(y_h,u_h)$ to the constructed solution $(\by,\bu)$ in dependence of the degrees of freedom is shown in Figure \ref{fig:Ex3-erroryu}.

\begin{figure}[H]
\begin{minipage}[h]{0.48\textwidth}
\includegraphics[width=\textwidth]{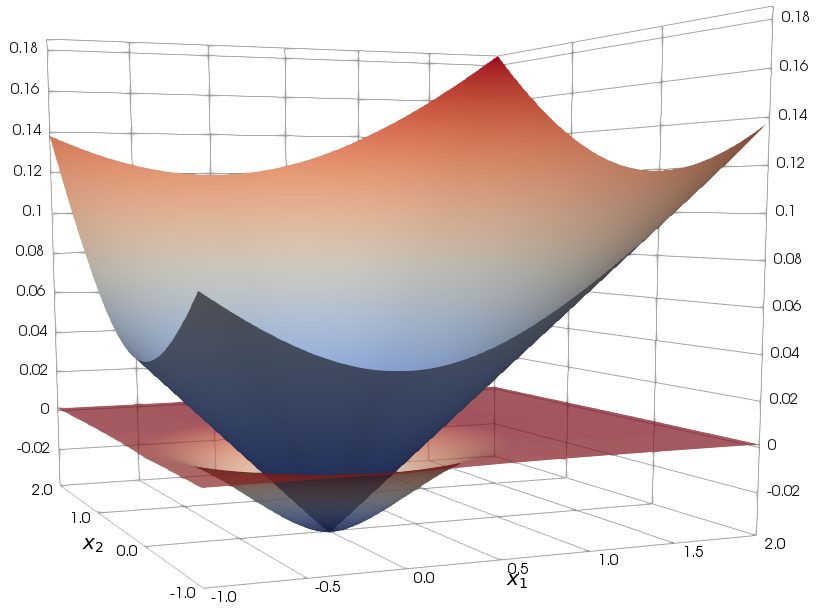}
\end{minipage}
\hfill
\begin{minipage}[h]{0.48\textwidth}
\includegraphics[width=\textwidth]{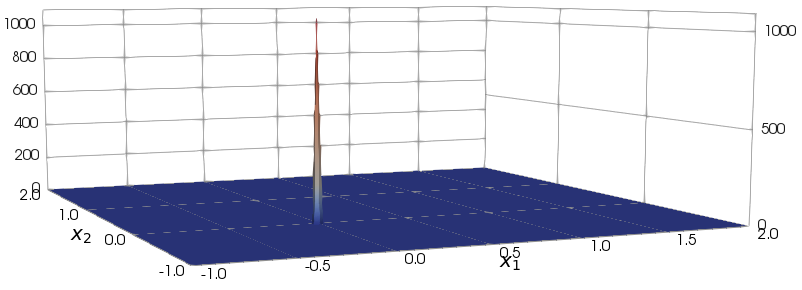}
\end{minipage}
\caption{(Example 3) Computed discrete optimal state $y_h$ with state constraint $\psi$ (left) and multiplier $\mu_h$ (right)}
\label{fig:Ex9-100-y}
\end{figure}
\begin{figure}[H]
\begin{minipage}[b]{0.48\textwidth}
\includegraphics[width=\textwidth]{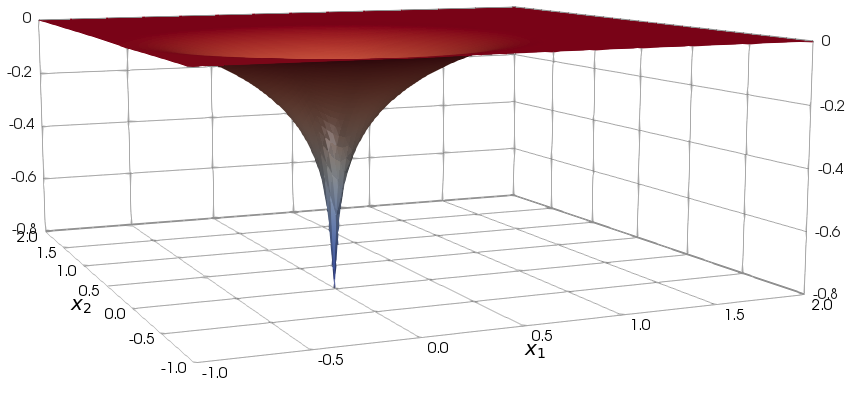}
\end{minipage}
\hfill
\begin{minipage}[b]{0.48\textwidth}
\includegraphics[width=\textwidth]{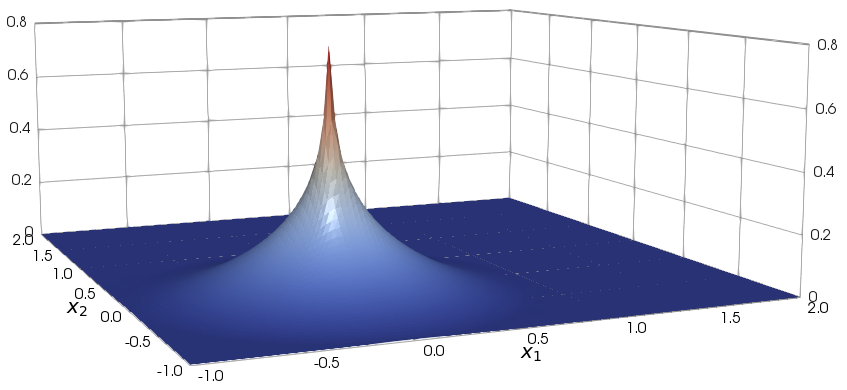}
\end{minipage}
\caption{(Example 3) Computed discrete optimal control $u_h$ (left) and the adjoint state $p_h$ (right)}
\label{fig:Ex9-100-mu}
\end{figure}

\begin{figure}[ht]
\begin{minipage}[b]{\textwidth}
\begin{figure}[H]
\centering
\setlength\figureheight{4cm}
\setlength\figurewidth{7cm}
\input{Ex3-erroryu.tikz}
\end{figure}
\end{minipage}
\caption{(Example 3): Errors $\norm{u_h-\bu}_{L^2(\Omega)}$ and $\norm{y_h-\by}_{L^2(\Omega)}$ vs. degrees of freedom.}
\label{fig:Ex3-erroryu}
\end{figure}

\subsection*{Penalization parameter and boundedness of the Lagrange multiplier}
Let us report about the the behaviour of the penalty parameter and the $L^1$-norm of the Lagrange multiplier.
Figure \ref{fig:L1andpenalty} depicts the $L^1$-norm of the computed multipliers $\mu_k$ and a scaled version of the penalty parameter $\rho_k$ during the iterations for all examples examined for a degree of freedom of $10^5$. In all cases, according to Lemma \ref{lemma:multL1_boundedness} the $L^1$-norm is clearly bounded. However we cannot recognize any indication of boundedness of the penalty parameter $\rho$. In fact $\rho$ seems to tend to infinity for all examples.

\begin{figure}[ht]
\centering
\setlength\figureheight{3cm}
\setlength\figurewidth{0.55\textwidth}
\input{L1andpenalty6-316-twoaxis.tikz}
\input{L1andpenalty8-225-twoaxis.tikz}
\input{L1andpenalty9-316-twoaxis.tikz}
\caption{$L^1(\Omega)$-norm of discrete multipliers $\mu_k$, penalty parameters $\rho_k$ vs.\@ iteration number for Example 1 (top), Example 2 (middle), Example 3 (bottom).}
\label{fig:L1andpenalty}
\end{figure}
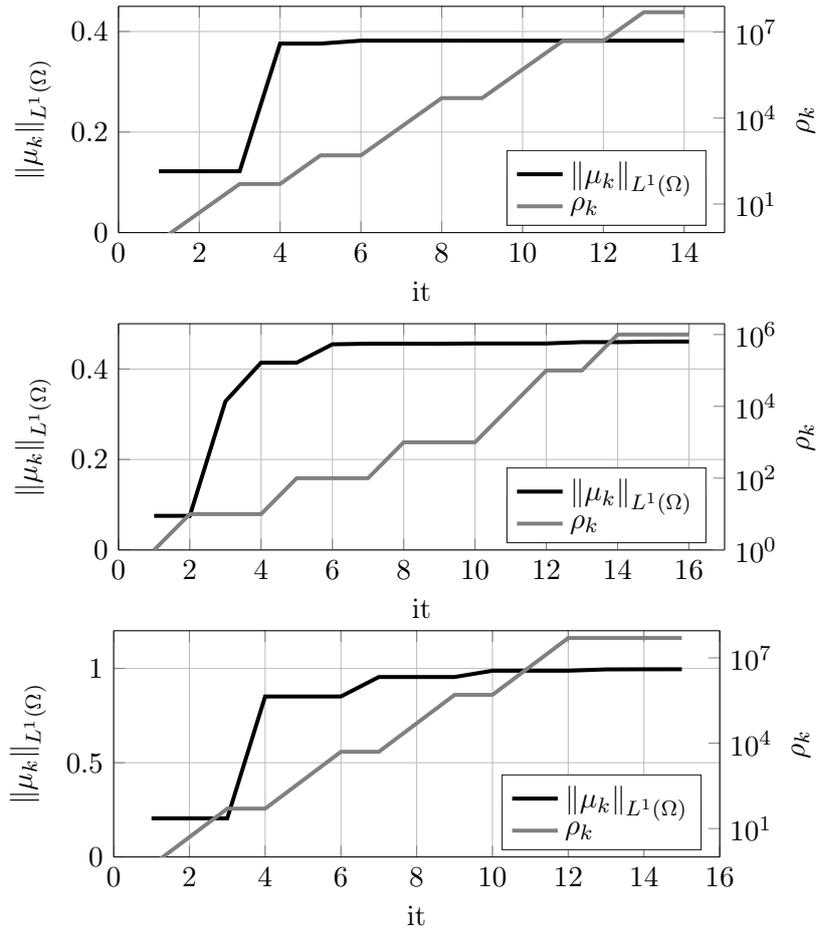

Finally, we want to give an overview about the number of iterations and the final penalization parameter for different refinements of the mesh in all examined examples.
Table \ref{table:it-rhomax} shows the number of outer iterations until the stopping criterion is reached as well as the accumulated inner iterations that are needed to solve the occuring sub-problems using an active set method.
Further, it represents the penalization parameter $\rho_{max}$ after the final iteration
and the $L^1$-norm of the approximated Lagrange multiplier. Table \ref{table:it-rhomax} indicates that a higher mesh refinement postulates a stronger penalization of the augmented Lagrange term in order to reach the stopping criterion.
\renewcommand{\arraystretch}{1.1}

\begin{table}[H]
\centering
\begin{tabularx}{\textwidth}{l cXXXX}
\hline
Degrees of freedom & & $10^2$ & $10^3$ & $10^4$ & $10^5$\\
\hline
	Example 1 %(NeumannRand)
	&it(outer) & 8	& 10	& 12	 & 14	\\
	&it(inner) & 18	& 31	& 37 	& 56	\\
	&$\rho_{max}$  & $5\cdot 10^{1}$ & 	$5\cdot 10^{3}	$	& $5\cdot 10^{5}$	 & $5\cdot 10^{7}$ \\
		&$\norm{\mu_h}_{L^1(\Omega)}$  & $3.8\cdot 10^{-1}$ 	& $3.8\cdot 10^{-1}$  	& $3.8\cdot 10^{-1}$& $3.8\cdot 10^{-1}$ \\
	\hline
		Example 2 %(selbst gebaut)
	&it(outer) & 12	& 14	& 15	 & 16	\\
	&it(inner) & 29	& 41	& 48 	& 71	\\
	&$\rho_{max}$  & $ 10^{3}$ & 	$ 10^{5}	$	& $10^{6}$	 & $10^{6}$ \\
		&$\norm{\mu_h}_{L^1(\Omega)}$  & $1.3$ 	& $3.5\cdot 10^{-1}$  	& $4.5\cdot 10^{-1}$& $4.6\cdot 10^{-1}$ \\
	\hline
		Example 3 %(Dirac)
	&it(outer) & 13	& 12	& 13	 & 15	\\
	&it(inner) & 26	& 32	& 51 	& 63	\\
	&$\rho_{max}$  & $ 5\cdot 10^{4}$ & 	$ 5\cdot 10^{5}	$	& $5\cdot 10^{6}$	 & $5\cdot 10^{7}$ \\
		&$\norm{\mu_h}_{L^1(\Omega)}$  & $7.9\cdot 10^{-1}$ 	& $9.2\cdot 10^{-1}$  	& $9.8\cdot 10^{-1}$& $9.96\cdot 10^{-1}$ \\
	\hline
\end{tabularx}
\caption{Iteration history for different discretizations}
\label{table:it-rhomax}
\end{table}

%bibliography
\bibliography{mybib}
\bibliographystyle{plain_abbrv}

\end{document}

%% file: Ex2-erroryu.tikz
\begin{tikzpicture}

\begin{loglogaxis}[%
width=0.951\figurewidth,
height=\figureheight,
at={(0\figurewidth,0\figureheight)},
scale only axis,
xmin=10,
xmax=1000000,
xminorticks=true,
yminorticks=true,
xlabel={degrees of freedom},
ylabel={error},
ymin=0.0001,
ymax=10,
grid=major,
axis background/.style={fill=white},
legend entries={$\norm{y_h-\by}_{L^2(\Omega)}$, $\norm{u_h-\bu}_{L^2(\Omega)}$,$\norm{p_h-\bp}_{L^2(\Omega)}$},
    legend pos=north east,
]
\addplot [color=black,solid,line width=1.5pt,mark size=1.8pt,mark=square*,mark options={solid,fill=black}, ]
  table[row sep=crcr]{%
100	0.3518\\
1000	0.03363\\
10000	0.002573\\
100000	0.0002544\\
};

\addplot [color=gray,solid,line width=1.5pt,mark size=2.5pt,mark=triangle*,mark options={solid,fill=gray,draw=gray}]
plot coordinates {%
(100	,3.494)
(1000	,0.3989)
(10000	,0.03741)
(100000	,0.003692)
};

\addplot [color=blue,dashed,line width=1.5pt,mark size=2.0pt,mark=*,mark options={solid,fill=blue,draw=blue}]
  table[row sep=crcr]{%
100	0.3672\\
1000	0.04278\\
10000	0.004335\\
100000	0.001964\\
};

\end{loglogaxis}
\end{tikzpicture}%

%% file: Ex3-erroryu.tikz
\begin{tikzpicture}

\begin{axis}[%
width=0.951\figurewidth,
height=\figureheight,
at={(0\figurewidth,0\figureheight)},
scale only axis,
xmode=log,
xmin=10,
xmax=1000000,
xminorticks=true,
xlabel={degrees of freedom},
%xmajorgrids,
%xminorgrids,
ymode=log,
ymin=0.001,
ymax=0.1,
yminorticks=true,
ylabel={error},
%ymajorgrids,
%yminorgrids,
grid=major,
axis background/.style={fill=white},
legend entries={$\norm{y_h-\by}_{L^2(\Omega)}$, $\norm{u_h-\bu}_{L^2(\Omega)}$},
    legend pos=north east,
]
\addplot [color=black,solid,line width=1.5pt,mark size=1.8pt,mark=square*,mark options={solid,fill=black}]
  table[row sep=crcr]{%
100	0.02688\\
1000	0.01422\\
10000	0.003842\\
100000	0.001226\\
};

\addplot [color=gray,solid,line width=1.5pt,mark size=2.5pt,mark=triangle*,mark options={solid,fill=gray,draw=gray}]
  table[row sep=crcr]{%
100	0.04026\\
1000	0.02703\\
10000	0.009102\\
100000	0.003187\\
};

\end{axis}
\end{tikzpicture}%

%% file: L1andpenalty6-316-twoaxis.tikz
\begin{tikzpicture}

\pgfplotsset{
width=\figurewidth,
height=\figureheight,
at={(0\figurewidth,0\figureheight)},
scale only axis,
xmin=0, xmax=15,
xmajorgrids,
legend style={at={(0.97,0.03)},anchor=south east,legend cell align=left,align=left,draw=white!15!black}
}

\begin{axis}[
  axis y line*=left,
  ymin=0, ymax=0.45,
  xlabel=it,
  ylabel= $\norm{\mu_k}_{L^1(\Omega)}$,
  ymajorgrids
]
\addplot [color=black,solid,line width=1.5pt]
  table[row sep=crcr]{%
1	0.1220501\\
2	0.1220501\\
3	0.1220501\\
4	0.3759769\\
5	0.3759769\\
6	0.3818116\\
7	0.3818116\\
8	0.3818116\\
9	0.3817370\\
10	0.3817370\\
11	0.3817370\\
12	0.3817297\\
13	0.3817297\\
14	0.3817294\\
};
\addlegendentry{${\|\mu_k\|}_{L^1(\Omega)}$};
\end{axis}

\begin{axis}[
  axis y line*=right,
  axis x line=none,
  ymin=1.0, ymax=80000000.0,
  ylabel= $\rho_k$,
  ymode = log,
]
\addlegendimage{/pgfplots/refstyle=plot_one}\addlegendentry{${\|\mu_k\|}_{L^1(\Omega)}$}
\addplot [color=gray,solid,line width=1.5pt]
  table[row sep=crcr]{%
1	0.5000000\\
2	5.0000000\\
3	50.0000000\\
4	50.0000000\\
5	500.0000000\\
6	500.0000000\\
7	5000.0000000\\
8	50000.0000000\\
9	50000.0000000\\
10	500000.0000000\\
11	5000000.0000000\\
12	5000000.0000000\\
13	50000000.0000000\\
14	50000000.0000000\\
};
\addlegendentry{$\rho_k$};
\end{axis}

\end{tikzpicture}

%% file: L1andpenalty8-225-twoaxis.tikz
\begin{tikzpicture}

\pgfplotsset{
width=\figurewidth,
height=\figureheight,
at={(0\figurewidth,0\figureheight)},
scale only axis,
xmin=0, xmax=17,
xmajorgrids,
legend style={at={(0.97,0.03)},anchor=south east,legend cell align=left,align=left,draw=white!15!black}
}

\begin{axis}[
  axis y line*=left,
  ymin=0, ymax=0.5,
  xlabel=it,
  ylabel= $\norm{\mu_k}_{L^1(\Omega)}$,
  ymajorgrids
]
\addplot [color=black,solid,line width=1.5pt]
  table[row sep=crcr]{%
1	0.0754758\\
2	0.0754758\\
3	0.3284037\\
4	0.4140554\\
5	0.4140554\\
6	0.4545461\\
7	0.4559566\\
8	0.4559566\\
9	0.4558518\\
10	0.4563952\\
11	0.4563952\\
12	0.4563952\\
13	0.4593534\\
14	0.4593534\\
15	0.4604717\\
16 0.4607016\\
};

\addlegendentry{${\|\mu_k\|}_{L^1(\Omega)}$};
\end{axis}

\begin{axis}[
  axis y line*=right,
  axis x line=none,
  ymin=1.0, ymax=2000000.0,
  ylabel= $\rho_k$,
  ymode = log
]
\addlegendimage{/pgfplots/refstyle=plot_one}\addlegendentry{${\|\mu_k\|}_{L^1(\Omega)}$}
\addplot [color=gray,solid,line width=1.5pt]
  table[row sep=crcr]{%
1	1.0000000 \\
2	10.0000000 \\
3	10.0000000\\
4	10.0000000\\
5	100.0000000\\
6	100.0000000\\
7	100.0000000\\
8	1000.0000000\\
9	1000.0000000\\
10	1000.0000000\\
11	10000.0000000\\
12	100000.0000000\\
13	100000.0000000\\
14	1000000.0000000\\
15	1000000.0000000\\
16	1000000.0000000\\
};
\addlegendentry{$\rho_k$};
\end{axis}

\end{tikzpicture}

%% file: L1andpenalty9-316-twoaxis.tikz
\begin{tikzpicture}

\pgfplotsset{
width=\figurewidth,
height=\figureheight,
at={(0\figurewidth,0\figureheight)},
scale only axis,
xmin=0, xmax=16,
xmajorgrids,
legend style={at={(0.97,0.03)},anchor=south east,legend cell align=left,align=left,draw=white!15!black}
}

\begin{axis}[
  axis y line*=left,
  ymin=0, ymax=1.2,
  xlabel=it,
  ylabel= $\norm{\mu_k}_{L^1(\Omega)}$,
  ymajorgrids
]
\addplot [color=black,solid,line width=1.5pt]
  table[row sep=crcr]{%
1	0.2049298\\
2	0.2049298\\
3	0.2049298\\
4	0.8513443\\
5	0.8513443\\
6	0.8513443\\
7	0.9545351\\
8	0.9545351\\
9	0.9545351\\
10	0.9879191\\
11	0.9879191\\
12	0.9879191\\
13	0.9949837\\
14	0.9955952\\
15	0.9957147\\
};
\label{plot_one}
\addlegendentry{${\|\mu_k\|}_{L^1(\Omega)}$};
\end{axis}

\begin{axis}[
  axis y line*=right,
  axis x line=none,
  ymin=1.0, ymax=90000000.0,
  ylabel= $\rho_k$,
  ymode = log
]
\addlegendimage{/pgfplots/refstyle=plot_one}\addlegendentry{${\|\mu_k\|}_{L^1(\Omega)}$}
\addplot [color=gray,solid,line width=1.5pt]
  table[row sep=crcr]{%
1	0.5000000\\
2	5.0000000\\
3	50.0000000\\
4	50.0000000\\
5	500.0000000\\
6	5000.0000000\\
7	5000.0000000\\
8	50000.0000000\\
9	500000.0000000\\
10	500000.0000000\\
11	5000000.0000000\\
12	50000000.0000000\\
13	50000000.0000000\\
14	50000000.0000000\\
15	50000000.0000000\\
};
\addlegendentry{$\rho_k$};
\end{axis}

\end{tikzpicture}